\documentclass[reqno,a4paper]{amsart}
\usepackage[utf8]{inputenc}
\usepackage[T1]{fontenc}
\usepackage[british]{babel}
\usepackage{amssymb,amsthm,amsmath,mathtools}
\usepackage{enumitem}
\usepackage[colorlinks]{hyperref}
\usepackage{microtype}
\usepackage{orcidlink}
\usepackage{doi}

\theoremstyle{plain}
\newtheorem{theorem}{Theorem}
\newtheorem{proposition}[theorem]{Proposition}
\newtheorem{lemma}[theorem]{Lemma}
\theoremstyle{remark}
\newtheorem{remark}[theorem]{Remark}

\newcommand{\defn}{\textbf}
\newcommand{\ot}{\leftarrow}
\newcommand{\tensor}{\otimes}
\newcommand{\F}{\ensuremath{\mathbb{F}}}
\newcommand{\V}{\ensuremath{\mathcal{V}}}
\newcommand{\E}{\ensuremath{\mathcal{E}}}
\newcommand{\D}{\ensuremath{\mathcal{D}}}
\newcommand{\AAlg}{\ensuremath{\mathbf{Assoc}}}
\newcommand{\CAAlg}{\ensuremath{\mathbf{CAssoc}}}
\newcommand{\JJord}{\ensuremath{\mathbf{JJord}}}
\newcommand{\PAlg}{\ensuremath{\mathbf{PAlg}}}
\DeclareMathOperator{\SplExt}{SplExt}
\DeclareMathOperator{\End}{End}
\DeclareMathOperator{\ADer}{ADer}
\DeclareMathOperator{\Hom}{Hom}
\DeclareMathOperator{\Imm}{Im}
\DeclareMathOperator{\M}{M}
\DeclareMathOperator{\kar}{char}
\DeclareMathOperator{\op}{op}

\title[Corrigendum to ``Weak repr.\ of actions of non-associative algebras'']{Corrigendum to\\ ``Weak representability of actions of non-associative algebras''\\ {[Journal of Algebra 669 (2025), 401--444]}}

\author[J.~Brox]{Jose Brox~\orcidlink{0000-0001-9822-5838}}
\author[X.~García-Martínez]{Xabier García-Martínez~\orcidlink{0000-0003-1679-4047}}
\author[M.~Mancini]{Manuel Mancini~\orcidlink{0000-0003-2142-6193}}
\author[T.~Van~der Linden]{Tim Van~der Linden~\orcidlink{0000-0001-5474-5356}}
\author[C.~Vienne]{Corentin Vienne~\orcidlink{0000-0002-0397-3495}}

\email{josebrox@uva.es}
\email{xabier.garcia@usc.gal}
\email{manuel.mancini@uclouvain.be; manuel.mancini@unipa.it}
\email{tim.vanderlinden@uclouvain.be; tim.van.der.linden@vub.be}
\email{corentin.vienne@proton.me}

\address[J.~Brox]{Departamento de Álgebra, Análisis Matemático, Geometría y Topología, Universidad de Valladolid, Palacio de Santa Cruz, E--47002 Valladolid, Spain.}
\address[X.~García-Martínez]{CITMAga \& Universidade de Santiago de Compostela, Departamento de
Matem\'aticas, E--15782 Santiago de Compostela, Spain}
\address[M.~Mancini, T.~Van der Linden]{Institut de Recherche en Mathématique et Physique, Université catholique de Louvain, chemin du cyclotron 2 bte L7.01.02, B--1348 Louvain-la-Neuve, Belgium.}
\address[M.~Mancini]{Dipartimento di Matematica e Informatica, Universit\`a degli Studi di Palermo, Via Archirafi 34, 90123 Palermo, Italy.}
\address[T.~Van~der Linden]{Mathematics and Data Science, Vrije Universiteit Brussel, Pleinlaan 2, B--1050 Brussel, Belgium}
\date{}


\begin{document}

\begin{abstract}
    Theorem~2.11 of the article, which states that the variety of commutative associative algebras over a field is weakly action representable, is false. Its proof rests on the claim that the argument of Theorem~2.9, which says that a span of commutative associative algebras with an amalgam in the category of associative algebras admits an amalgam in the category of commutative associative algebras, extends to arbitrary diagrams. We show that this is not the case: for a certain \(16\)-dimensional vector space \(X\) with zero multiplication, the diagram formed by five commutative subalgebras of \(\End(X)\) and their intersections admits an amalgam in the category of associative algebras but none in the category of commutative ones, and it follows that the functor \(\SplExt(-,X)\) admits no weak representation. We also correct Proposition~2.8, whose ``if'' part needs a strong amalgam, we update the definition of partial algebra, and we list the statements of the article which are affected.
\end{abstract}

\maketitle

Throughout, \(\F\) is a field, all algebras are \(\F\)-algebras, and we use the notation and the numbering of~\cite{WRA}. In particular, \(\CAAlg\) and \(\AAlg\) are the varieties of commutative associative and of associative algebras, \(U\colon\CAAlg\to\AAlg\) is the forgetful functor, and for a commutative associative algebra \(X\) we write \(\M(X)\) for its associative algebra of multipliers, which is the algebra \(\End(X)\) of all linear endomorphisms of \(X\) when the multiplication of \(X\) is zero.

Theorem~2.11 of~\cite{WRA} states that \(\CAAlg\) is weakly action representable. This is false: in Section~\ref{sec:counterexample} we exhibit an algebra \(X\) of dimension \(16\) with zero multiplication such that the functor \(\SplExt(-,X)\) admits no weak representation. Proposition~2.8 of~\cite{WRA} reduces the existence of a weak representation of \(\SplExt(-,X)\) to the existence of an amalgam in \(\CAAlg\) for the diagram of all commutative subalgebras of \(\M(X)\), and Theorem~2.9 says that a span \(S\ot I\to T\) of commutative associative algebras which admits an amalgam in \(\AAlg\) admits one in \(\CAAlg\). Theorem~2.9 is correct. The error is in the sentence which follows its proof, claiming that ``for arbitrary diagrams of monomorphisms of commutative associative algebras, the proof stays essentially the same''. This is not so: the proof uses the commutativity of the algebras \(S\) and \(T\) in an essential way and does not extend beyond spans, and Section~\ref{sec:counterexample} shows that its conclusion may fail for diagrams whose objects overlap along a cycle. We explain this in Section~\ref{sec:proof}, where we also point out that the proof of the ``if'' part of Proposition~2.8 needs a slightly stronger hypothesis than the one stated. Section~\ref{sec:amendments} lists the statements of~\cite{WRA} which have to be amended as a consequence, and Section~\ref{Further corrections} collects some further corrections, including a modification of the definition of partial algebra. The rest of the article is not affected.

\section{The counterexample}\label{sec:counterexample}

Let \(X\) be the vector space of dimension \(16\) with basis
\begin{gather*}
    e,\qquad
    e_{A},\ e_{B},\ e_{C},\ e_{D},\qquad
    e_{AB},\ e_{BD},\ e_{CA},\ e_{DA},\ e_{BC},\ e_{DC},\\
    e_{DAB},\ e_{BDC},\ e_{ADC},\ e_{BCA},\qquad
    e_{ABDC}\text{,}
\end{gather*}
regarded as an object of \(\CAAlg\) with zero multiplication, so that \(\M(X)=\End(X)\). The subscripts are mere names, chosen so that the definitions below are easy to remember: \(e_{w}\) is the image of \(e\) under the word \(w\), the letters acting as the endomorphisms defined next and the rightmost letter acting first. We define four endomorphisms \(a\), \(b\), \(c\), \(d\) of \(X\) by the following table, in which every basis vector not listed in a row is sent to \(0\) by the corresponding endomorphism:
\begin{align*}
    a\colon\  & e\mapsto e_{A},\quad e_{B}\mapsto e_{AB},\quad e_{DC}\mapsto e_{ADC},\quad e_{BDC}\mapsto e_{ABDC}\text{,}          \\
    b\colon\  & e\mapsto e_{B},\quad e_{A}\mapsto e_{AB},\quad e_{C}\mapsto e_{BC},\quad e_{D}\mapsto e_{BD},                       \\
              & e_{CA}\mapsto e_{BCA},\quad e_{DA}\mapsto e_{DAB},\quad e_{DC}\mapsto e_{BDC},\quad e_{ADC}\mapsto e_{ABDC}\text{,} \\
    c\colon\  & e\mapsto e_{C},\quad e_{A}\mapsto e_{CA},\quad e_{BD}\mapsto e_{BDC},\quad e_{DAB}\mapsto e_{ABDC}\text{,}          \\
    d\colon\  & e\mapsto e_{D},\quad e_{A}\mapsto e_{DA},\quad e_{B}\mapsto e_{BD},\quad e_{C}\mapsto e_{DC},                       \\
              & e_{AB}\mapsto e_{DAB},\quad e_{CA}\mapsto e_{ADC},\quad e_{BC}\mapsto e_{BDC},\quad e_{BCA}\mapsto e_{ABDC}\text{.}
\end{align*}
Each of the four is a partial injection of the basis into itself, so that composites are again partial maps of the basis and are computed by following arrows. We write \(ab\) for \(a\circ b\), and so on.

\begin{lemma}\label{lem:relations}
    The endomorphisms \(a\), \(b\), \(c\), \(d\) satisfy
    \[
        ab=ba,\quad (ab)c=c(ab)=0,\quad bd=db,\quad (bd)c=c(bd),\quad a(bdc)=(bdc)a\text{,}
    \]
    and moreover \(abcd=0\) while \(abdc\neq0\).
\end{lemma}

\begin{proof}
    Following the table, \(ab\) and \(ba\) both send \(e\) to \(e_{AB}\) and \(e_{DC}\) to \(e_{ABDC}\), and kill the other basis vectors. For instance, \(b\) sends \(e_{DC}\) to \(e_{BDC}\), which \(a\) sends to \(e_{ABDC}\), while \(a\) sends \(e_{DC}\) to \(e_{ADC}\), which \(b\) sends to \(e_{ABDC}\). Every other basis vector is killed by \(ab\), because \(b\) either kills it or sends it outside the domain of~\(a\), and by \(ba\) for the symmetric reason. The image of \(c\) is spanned by \(e_{C}\), \(e_{CA}\), \(e_{BDC}\) and \(e_{ABDC}\), none of which is \(e\) or \(e_{DC}\), so that \((ab)c=0\). The image of \(ab\) is spanned by \(e_{AB}\) and \(e_{ABDC}\), both killed by \(c\), so that \(c(ab)=0\). In the same way, \(bd\) and \(db\) both send \(e\), \(e_{A}\), \(e_{C}\), \(e_{CA}\) to \(e_{BD}\), \(e_{DAB}\), \(e_{BDC}\), \(e_{ABDC}\) respectively, and kill the remaining basis vectors. Consequently \((bd)c\) and \(c(bd)\) both send \(e\) to \(e_{BDC}\) and \(e_{A}\) to \(e_{ABDC}\) and kill everything else, and finally \(a(bdc)\) and \((bdc)a\) both send \(e\) to \(e_{ABDC}\) and kill everything else. The last assertion is now clear: \(abcd=(abc)d=0\), whereas \(abdc=a((bd)c)\) sends \(e\) to \(e_{ABDC}\).
\end{proof}

Consider the five subalgebras
\[
    S_{1}\coloneqq\langle a,b\rangle,\quad
    S_{2}\coloneqq\langle ab,c\rangle,\quad
    S_{3}\coloneqq\langle b,d\rangle,\quad
    S_{4}\coloneqq\langle bd,c\rangle,\quad
    S_{5}\coloneqq\langle a,bdc\rangle
\]
of \(\End(X)\), where \(\langle f,g\rangle\) denotes the subalgebra generated by \(f\) and \(g\). By Lemma~\ref{lem:relations}, each of them is generated by two commuting endomorphisms, hence is a commutative subalgebra of \(\End(X)\). We shall use the memberships \(a\in S_{1}\cap S_{5}\), \(b\in S_{1}\cap S_{3}\), \(c\in S_{2}\cap S_{4}\), \(ab\in S_{1}\cap S_{2}\), \(bd\in S_{3}\cap S_{4}\) and \(bdc\in S_{4}\cap S_{5}\), all clear from the definitions, together with \(abdc\in S_{5}\). Let \(\D\) be the diagram in \(\CAAlg\) formed by \(S_{1}\), \dots, \(S_{5}\) and their pairwise intersections, with the inclusions between them.

\begin{proposition}\label{prop:cocone}
    Let \(T\) be a commutative associative algebra and let \(\mu_{i}\colon S_{i}\to T\), for \(i=1\), \dots, \(5\), be homomorphisms of algebras such that \(\mu_{i}\) and \(\mu_{j}\) agree on \(S_{i}\cap S_{j}\) for all \(i\), \(j\). Then \(\mu_{5}(abdc)=0\). Consequently, the diagram \(\D\) admits an amalgam in \(\AAlg\), but none in \(\CAAlg\).
\end{proposition}

\begin{proof}
    Because of the compatibility on the intersections we may write \(\mu\) for the common value of the \(\mu_{i}\) on any element of \(\bigcup_{i}S_{i}\) without ambiguity. Since \(\mu_{1}\) is a homomorphism, \(\mu(a)\mu(b)=\mu(ab)\), and since \(\mu_{2}\) is a homomorphism and \((ab)c=0\) by Lemma~\ref{lem:relations},
    \[
        \mu(a)\mu(b)\mu(c)=\mu(ab)\mu(c)=\mu_{2}\bigl((ab)c\bigr)=0\text{.}
    \]
    On the other hand, using \(\mu_{5}\), \(\mu_{4}\) and \(\mu_{3}\) in turn,
    \[
        \mu(abdc)=\mu(a)\mu(bdc)=\mu(a)\bigl(\mu(bd)\mu(c)\bigr)=\mu(a)\bigl((\mu(b)\mu(d))\mu(c)\bigr)\text{.}
    \]
    In the commutative associative algebra \(T\), the right-hand side is a product of the four elements \(\mu(a)\), \(\mu(b)\), \(\mu(c)\) and \(\mu(d)\), so it may be rewritten as \(\mu(a)\mu(b)\mu(c)\) times \(\mu(d)\), and it vanishes.

    The inclusions into \(\End(X)\) form a monic cocone on \(\D\) in \(\AAlg\), which is an amalgam. If \((\mu_{S}\colon S\to T)_{S\in\D}\) is a cocone on \(\D\) in \(\CAAlg\), then \(\mu_{S_{i}}\) and \(\mu_{S_{j}}\) both restrict to \(\mu_{S_{i}\cap S_{j}}\), so that the first part applies and \(\mu_{S_{5}}\) sends the nonzero element \(abdc\) of \(S_{5}\) to \(0\). Hence no cocone on \(\D\) in \(\CAAlg\) is monic.
\end{proof}

\begin{theorem}\label{thm:main}
    The functor \(\SplExt_{\CAAlg}(-,X)\) admits no weak representation. In particular, the variety \(\CAAlg\) of commutative associative algebras over \(\F\) is not weakly action representable.
\end{theorem}

\begin{proof}
    Suppose that \((T,\tau)\) is a weak representation of \(\SplExt(-,X)\). By Lemma~2.6 of~\cite{WRA} there is an injective function \(j\colon\End(X)\to T\) such that \(j\circ\rho_{B}(\xi)=\tau_{B}(\xi)\) for every commutative associative algebra \(B\) and every \(\xi\in\SplExt(B,X)\), where \(\rho\) is the natural isomorphism of Lemma~2.4. Let \(S\) be a commutative subalgebra of \(\End(X)\). The inclusion \(\iota_{S}\colon U(S)\to\End(X)\) is a morphism of \(\AAlg\), so it is \(\rho_{S}(\xi_{S})\) for a unique \(\xi_{S}\in\SplExt(S,X)\), and \(\tau_{S}(\xi_{S})=j\circ\iota_{S}\) is a morphism \(S\to T\) of \(\CAAlg\). In other words, the restriction of \(j\) to any commutative subalgebra of \(\End(X)\) is a homomorphism of algebras. The restrictions \(\mu_{i}\coloneqq j|_{S_{i}}\) to the subalgebras \(S_{1}, \dots,S_{5}\) are then homomorphisms which trivially agree on the intersections \(S_{i}\cap S_{j}\), so Proposition~\ref{prop:cocone} gives \(j(abdc)=\mu_{5}(abdc)=0=\mu_{5}(0)=j(0)\). Since \(abdc\neq0\) by Lemma~\ref{lem:relations}, this contradicts the injectivity of \(j\).
\end{proof}

\begin{remark}\label{rem:mechanism}
    The proof works over any field, and the table which defines \(a\), \(b\), \(c\), \(d\) has entries \(0\) and \(1\) only. The mechanism is the following. In the proof of Proposition~\ref{prop:cocone}, the commutativity of~\(T\) is used only once, to interchange \(\mu(d)\) and \(\mu(c)\). In \(\End(X)\), however, \(c\) and \(d\) do not commute: \(dc\) sends \(e\) to \(e_{DC}\), whereas \(cd\) kills \(e\). All the commutations imposed by the diagram hold in \(\End(X)\), but they do not force \(c\) and \(d\) to commute, and this is what keeps \(abdc\) alive while \(abc\), hence \(abcd\), vanishes. In a weak actor, which is commutative, the images of the four endomorphisms do commute, so that the image of \(abdc\) is forced to coincide with the product of the images of \(abc\) and~\(d\), which is~\(0\). The example is thus a failure of amalgamation: the commutativity available inside each~\(S_{i}\) does not propagate to their union in \(\AAlg\), but it must propagate in any cocone in \(\CAAlg\).
\end{remark}

\section{What goes wrong in the published argument}\label{sec:proof}

\subsection*{Amalgams of spans versus amalgams of diagrams}

We recall the proof of Theorem~2.9 of~\cite{WRA}. For a span \(S\ot I\to T\) of commutative associative algebras, one compares the pushout \(S+_{I}T=(S+T)/J\) in \(\AAlg\) with the pushout \(S+^{\CAAlg}_{I}T=(S+^{\CAAlg}T)/K\) in \(\CAAlg\), where \(J\) and \(K\) are the ideals generated by the elements \(g=\underline{i}-\overline{i}\) for \(i\in I\), and one uses the linear splitting \(\sigma\colon S+^{\CAAlg}T=S\oplus T\oplus(S\tensor T)\to S+T\) of the reflection which sends \(s\tensor t\) to the word \(st\). The heart of the proof is the inclusion \(\sigma(K)\subseteq J\), and the calculation which establishes it uses commutativity twice, as is pointed out there. For instance, the equality \(\sigma((0,t,0)g)=(0,-it,i\tensor t)=(i,-i,0)(0,t,0)\) requires that \(ti=it\) in \(T\), where \(t\) and \(i\) are viewed as elements of \(T\).

For a diagram with more objects, the same calculation produces terms which need not lie in \(J\). Consider three commutative associative algebras \(S_{1}\), \(S_{2}\), \(S_{3}\) with an element \(i\in S_{1}\cap S_{3}\), so that \(g=\underline{i}-\overline{i}\) with \(\underline{i}\in S_{1}\) and \(\overline{i}\in S_{3}\) is one of the generators. The coproduct in \(\CAAlg\) now contains \(S_{2}\tensor S_{3}\) and \(S_{1}\tensor S_{2}\tensor S_{3}\), and letting \(\sigma\) send \(s_{1}\tensor s_{2}\tensor s_{3}\) to the word \(s_{1}s_{2}s_{3}\), for \(s_{2}\in S_{2}\), \(s_{3}\in S_{3}\) we find
\[
    \sigma\bigl((s_{2}\tensor s_{3})g\bigr)=\underline{i}\,s_{2}s_{3}-s_{2}s_{3}\,\overline{i}\equiv\overline{i}\,s_{2}s_{3}-s_{2}s_{3}\,\overline{i}=[\,\overline{i},s_{2}\,]\,s_{3}\pmod J\text{,}
\]
where we used that \(\overline{i}\) and \(s_{3}\) commute in \(S_{3}\). Nothing forces the commutator of an element of \(S_{3}\) with an element of \(S_{2}\) into~\(J\): in the amalgam in \(\AAlg\), elements of different objects of the diagram need not commute. This already happens for the algebras \(S_{1}\), \(S_{2}\), \(S_{3}\) of the diagram~\(\D\) of Section~\ref{sec:counterexample}, with \(i=b\in S_{1}\cap S_{3}\), \(s_{2}=c\) and \(s_{3}=d\). Since \(\End(X)\) is a cocone on the diagram formed by these three algebras and their intersections, \(J\) is contained in the kernel of the induced homomorphism \(S_{1}+S_{2}+S_{3}\to\End(X)\). This homomorphism sends \(\sigma((c\tensor d)g)\) to \(bcd-cdb\), which equals \([b,c]\,d\) because \(bd=db\), and which sends \(e\) to \(-e_{BDC}\). Hence \(\sigma((c\tensor d)g)\notin J\) although \((c\tensor d)g\in K\), so that the inclusion \(\sigma(K)\subseteq J\) fails, and the conclusion \(S\cap\sigma(K)=0\) can no longer be drawn from \(S\cap J=0\). Proposition~\ref{prop:cocone} shows that this is not just a defect of the proof: the reflection \(\AAlg\to\CAAlg\) preserves amalgams of spans, by Theorem~2.9, but it does not preserve amalgams of arbitrary diagrams.

\subsection*{Amalgams versus strong amalgams}

The ``if'' part of the proof of Proposition~2.8 of~\cite{WRA} contains an unjustified step. Given an amalgam \((T,(\mu_{S})_{S})\) of the diagram of commutative subalgebras of \(\M(X)\), the component \(\tau_{B}(\xi)\) is defined as the composite of the surjection \(\rho'_{B}(\xi)\colon B\to M_{\xi}\) with \(\mu_{M_{\xi}}\), and its injectivity in \(\xi\) is argued as follows: if \(\tau_{B}(\xi)=\tau_{B}(\psi)\), then the images of \(\mu_{M_{\xi}}\) and \(\mu_{M_{\psi}}\) in \(T\) coincide, and since ``the image in \(\AAlg\) of the diagram \((M_{\xi})_{\xi}\) is a thin category'', \(M_{\xi}=M_{\psi}\). This last inference requires that distinct subalgebras have distinct images under the cocone, which is not part of the definition of an amalgam: a monic cocone is injective on each object of the diagram, not necessarily on their union \(\M(X)\). What the ``only if'' part of the proof actually produces, via Lemma~2.6, is a cocone which is injective on the union, and the correct equivalence is the following.

\begin{proposition}\label{prop:corrected}
    For a commutative associative algebra \(X\), \(\SplExt(-,X)\) admits a weak representation if and only if there exist a commutative associative algebra~\(T\) and an injective function \(j\colon\M(X)\to T\) whose restriction to every commutative subalgebra of \(\M(X)\) is a homomorphism of algebras.
\end{proposition}

\begin{proof}
    Necessity is Lemma~2.6 of~\cite{WRA} combined with the argument in the proof of Theorem~\ref{thm:main}. Conversely, given \((T,j)\), define \(\tau_{B}(\xi)\coloneqq j\circ\rho_{B}(\xi)\) for \(\xi\in\SplExt(B,X)\). This is a morphism of \(\CAAlg\), being the composite of the surjection of \(B\) onto the commutative subalgebra \(M_{\xi}\) of \(\M(X)\) with the restriction of \(j\) to \(M_{\xi}\). It is injective in \(\xi\) because \(j\) is injective and \(\rho_{B}\) is a bijection, and it is natural in \(B\) because \(\rho\) is. Hence \((T,\tau)\) is a weak representation.
\end{proof}

By analogy with the strong amalgamation property of~\cite{KISS_MARKI_PROHLE_THOLEN:1982}, we say that a cocone \((\mu_{S}\colon S\to T)_{S}\) on a diagram of subalgebras of a given algebra or partial algebra is a \defn{strong amalgam} when it is injective on the union of the diagram, that is, when \(\mu_{S}(s)=\mu_{S'}(s')\) implies \(s=s'\) for all \(s\in S\) and \(s'\in S'\). For an amalgam of a diagram which is closed under intersections, this means that the images of two objects intersect in the image of their intersection. Proposition~\ref{prop:corrected} says that \(\SplExt(-,X)\) admits a weak representation if and only if the diagram of commutative subalgebras of \(\M(X)\) admits a strong amalgam in \(\CAAlg\). Since any cocone on a diagram factors through the colimit cocone, a diagram which admits a strong amalgam has a colimit cocone which is one. Theorem~3.9(4) of~\cite{WRA} is obtained by copying the ``if'' part of the proof of Proposition~2.8, and it is to be corrected in the same way: if the diagram in \(\V\) given by those subalgebras of \(\E(X)\) which occur as codomain of a morphism in \(\Imm(\tau_{B})\) admits a strong amalgam in~\(\V\), then the colimit of that diagram determines an (initial) weak representation of \(\SplExt(-,X)\). The proof is that of Proposition~\ref{prop:corrected}. If \(\mu_{M_{\xi}}\circ\rho'_{B}(\xi)=\mu_{M_{\psi}}\circ\rho'_{B}(\psi)\), where now \(\rho'_{B}(\xi)\colon B\to M_{\xi}\) denotes the corestriction of \(\tau_{B}(\xi)\) to its image~\(M_{\xi}\), then injectivity on the union gives \(\tau_{B}(\xi)(b)=\tau_{B}(\psi)(b)\) for all \(b\in B\), hence \(\xi=\psi\) by Theorem~3.9(1). For the algebra \(X\) of Section~\ref{sec:counterexample} the distinction is immaterial, since by Proposition~\ref{prop:cocone} not even an amalgam exists.

\section{Amendments}\label{sec:amendments}

The following statements of~\cite{WRA} have to be modified.

\begin{itemize}[leftmargin=*]
    \item The claim that \(\CAAlg\) is weakly action representable has to be removed from the abstract, from the fifth paragraph of the introduction and from Section~2 of~\cite{WRA}, where it is announced before Lemma~2.4 (``We want to extend this result proving that the variety \(\CAAlg\) is weakly action representable'') and after Theorem~2.5 (``We now want to prove that it is a weakly action representable category. We analyse what this means and then prove that the category does indeed fulfil these requirements''). By Theorem~\ref{thm:main}, the variety \(\CAAlg\) is not weakly action representable.
    \item In Section~2 of~\cite{WRA}, the two sentences preceding Theorem~2.9, from ``We are actually going to prove something a bit stronger'' to ``\dots\ the following special case'', and the two sentences following its proof, ``For arbitrary diagrams of monomorphisms of commutative associative algebras, the proof stays essentially the same. This allows us to conclude:'', have to be removed. Theorem~2.11 is false, so that Remark~2.13 no longer applies to \(\CAAlg\). The general facts recalled there remain valid, namely that the weak representations constructed as colimits are initial, and that, by~\cite[Corollaries~5.3 and~5.4]{WAR}, a variety of non-associative algebras is weakly action representable if and only if it is initially weakly action representable. They are used in the proof of Theorem~3.9(4) and in Section~5.3. Theorem~2.9 and Lemma~2.10 are correct as stated.
    \item Remark~2.7 has to be replaced by the following observation. By the proof of Theorem~\ref{thm:main}, the function~\(j\) of Lemma~2.6 restricts to a homomorphism of algebras on each commutative subalgebra of \(\M(X)\), but it cannot be a homomorphism on all of \(\M(X)\) unless \(\M(X)\) is commutative. For the algebra \(X\) of Section~\ref{sec:counterexample}, no injective function from \(\M(X)\) to a commutative associative algebra is a homomorphism on each commutative subalgebra of \(\M(X)\). The first sentence of Remark~2.7, according to which~\(j\) is a vector space monomorphism, is false in general: the argument in the proof of Theorem~\ref{thm:main} only shows that \(j\) is linear on each commutative subalgebra of \(\M(X)\). For instance, for \(X=\F^{2}\) with zero multiplication, the tensor product of the distinct algebras \(\F1+\F A\), with \(A\in\End(\F^{2})\) non-scalar, gives a weak representation by Proposition~\ref{prop:corrected} whose function~\(j\) is not additive.
    \item In Proposition~2.8, in the sentence following its proof (``Thus we see that the problem of weak representability of actions of \(\CAAlg\) amounts to proving that an amalgam in \(\CAAlg\) exists \dots'') and in Theorem~3.9(4), ``amalgam'' has to be replaced by ``strong amalgam'', as explained in Section~\ref{sec:proof} above. The ``only if'' part of Proposition~2.8 is correct as stated.
    \item The question in Section~5.1 has already been answered in the negative in~\cite{XabiMancini}, where it is shown that the varieties of \(k\)-nilpotent Lie algebras, for \(k\geq3\), and of \(n\)-solvable Lie algebras, for \(n\geq2\), are action accessible but not weakly action representable. The variety \(\CAAlg\) is a further example: it is action accessible, since it satisfies the \(\lambda/\mu\)-rules (Theorem~1.9), but by Theorem~\ref{thm:main} it is not weakly action representable. Since \(\CAAlg\) is an operadic quadratic variety, the implication ``weakly action representable \(\Rightarrow\) action accessible'' cannot be reversed even for operadic quadratic varieties of algebras over a field. Concerning Section~5.2, the variety \(\AAlg\) is weakly action representable~\cite{WAR} while its subvariety \(\CAAlg\) is not, so that the condition (WRA) is not inherited by subvarieties, not even in the quadratic case. In the non-quadratic case, this also follows from~\cite{XabiMancini}, since the variety of Lie algebras is action representable.
\end{itemize}
No other result of~\cite{WRA} depends on Theorem~2.11 or on the ``if'' part of Proposition~2.8. In particular, the results of Sections~2.2--2.5 of~\cite{WRA}, on Jacobi--Jordan algebras, on two-step nilpotent (anti-)commutative algebras and on anti-commutative anti-associative algebras, as well as Sections~3 and~4 of~\cite{WRA} with the exception of Theorem~3.9(4), are not affected by these errors. The further corrections collected in Section~\ref{Further corrections} below are independent of them.

\section{Further corrections of scope and notation}\label{Further corrections}

\begin{enumerate}[label=(\alph*)]
    \item Theorem~1.10 relies on the classification of~\cite{Edinburgh}, which holds over an infinite field of characteristic different from~\(2\), as recalled in the first paragraph of the introduction of~\cite{WRA}. Hence, the hypothesis that \(\F\) is infinite and of characteristic different from~\(2\) has to be added to the statement of Theorem~1.10.

    \item In Proposition~2.15, condition~(2) has to read \(b\ast(xx')=-(b\ast x)x'-(b\ast x')x\).

    \item In Remark~2.16, if \(X=\F\) is the abelian one-dimensional algebra, then the anti-commutator on \(\ADer(X)=\End(X)\cong\F\) is given by \(\langle\alpha,\beta\rangle=-2\alpha\beta\), so that
          \(\langle\langle\alpha,\beta\rangle,\gamma\rangle+\langle\langle\beta,\gamma\rangle,\alpha\rangle+\langle\langle\gamma,\alpha\rangle,\beta\rangle=12\alpha\beta\gamma\).
          Hence this example exhibits a failure of the Jacobi identity only when \(\kar\F\notin\{2,3\}\), and Remark~2.16 has to be restricted to that case. If \(\kar\F\in\{2,3\}\), then \(\ADer(\F)\) is a Jacobi--Jordan algebra.

    \item In Theorem~2.17, the forgetful functor should be written \(U\colon{\JJord\to\PAlg}\), not \(U\colon\JJord^{\op}\to\PAlg\). The contravariance occurs only after applying the functor \(\Hom_{\PAlg}(U(-),\ADer(X))\).

    \item In Definitions~1.12 and~1.13 of~\cite{WRA}, a partial algebra is a vector space~\(X\) endowed with a bilinear partial operation \(\Omega\to X\), whose domain \(\Omega\) is a vector subspace of \(X\times X\). In all the constructions of~\cite{WRA}, the products with~\(0\) are defined, that is, \((x,0),(0,y)\in\Omega\) for all \(x,y\in X\). Since \(\Omega\) is a vector subspace, this gives \((x,y)=(x,0)+(0,y)\in\Omega\), so that \(\Omega=X\times X\) and the partial operation is necessarily total. To allow non-total multiplications, in~\cite[Section~4.21]{SurveyActions} the domain of the multiplication is instead taken to be a vector subspace of \(X\tensor X\): a \defn{partial algebra} is a triple \((X,D,\mu)\), where~\(X\) is a vector space, \(D\) is a vector subspace of \(X\tensor X\) and \(\mu\colon D\to X\) is a linear map, and the product \(xy\) is defined if and only if \(x\tensor y\in D\), in which case \(xy=\mu(x\tensor y)\). A homomorphism of partial algebras \((X,D,\mu)\to(X',D',\mu')\) is a linear map \(f\colon X\to X'\) such that \((f\tensor f)(D)\subseteq D'\) and \(f\circ\mu=\mu'\circ(f\tensor f)|_{D}\). The notion of identity of Definition~1.14 of~\cite{WRA} is adapted accordingly in~\cite[Section~4.21]{SurveyActions}. All the results of~\cite{WRA} which involve partial algebras, in particular Theorems~2.17 and~2.20 and all the results on the external weak actor \(\E(X)\), remain valid, with the same statements, for this definition of partial algebra. Proofs adapted to the new definition can be found in~\cite{SurveyActions}.

    \item In Example~4.2, the third identity in~(4.1) has to read \(x(fy)=-x(yf)\), as in the fifth row of the reduced matrix displayed there. Consequently, the value \((1,0,0,0)\) of \((\alpha_{1},\alpha_{2},\beta_{1},\beta_{2})\) in Example~4.3 has to be replaced by \((-1,0,0,0)\), which gives the standard multiplication of Example~3.5. Indeed, for \((1,0,0,0)\) and the one-dimensional abelian algebra~\(X\), the element \(u=(1,0)\) of \(\E(X)\cong\F^{2}\) satisfies \(u(uu)=4u\), so that the right Leibniz identity fails when \(\kar\F\neq2\).

    \item In Example~4.6, the identities \((xf)y=0\) and \(x(fy)=0\) of Definition~3.3 are missing from the displayed description of \(\E(X)\), so that the displayed space may be strictly larger than \(\E(X)\). This does not affect the conclusion, since any subspace of \(\End(X)^{2}\) containing \(\E(X)\), with the zero multiplication, is a weak actor of~\(X\).

    \item Example~4.7 relies on~\cite{MaRe}, where the result is proved in characteristic zero. Hence, the hypothesis \(\kar\F=0\) has to be added to Example~4.7.
\end{enumerate}

\section*{Acknowledgements}
The first author is supported by a postdoctoral fellowship ``Convocatoria 2021'' funded by Universidad de Valladolid, and partially supported by grant PID2022-137283NB-C22 funded by MCIN/AEI/10.13039/501100011033 and ERDF ``A way of making Europe''. The second author is supported by grant PID2024-155502NB-I00 funded by MICIU/AEI/10.13039/501100011033 and by FEDER, UE, and by Xunta de Galicia through the Competitive Reference Groups (GRC), ED431C 2023/31. The third author is supported by the ``National Group for Algebraic and Geometric Structures and their Applications'' (GNSAGA -- INdAM), and he is a Postdoctoral Researcher of the Fonds de la Recherche Scientifique--FNRS. The fourth author is a Research Director of the Fonds de la Recherche Scientifique--FNRS. The fifth author was supported by Wallonie-Bruxelles International.

%\section*{Declarations}
%The authors declare no competing interests.

\section*{Declaration of generative AI and AI-assisted technologies in the manuscript preparation process}
The counterexample leading to the present corrigendum was found by accident when one of the authors used generative AI (Anthropic's Claude Fable 5.1) in exploring closely related mathematics. During the preparation of this work, the authors used Anthropic's Claude Fable 5.1 to polish human-written notes, GPT~6 Sol to draft a first version of Section~\ref{Further corrections}, and Anthropic's Claude Opus 5.5 to check the corrigendum against the original article and to revise parts of the text. After using these tools, the authors reviewed and edited the content as needed and take full responsibility for the content of this corrigendum.

% \bibliographystyle{amsplain-nodash}
% \bibliography{tim}

\begin{thebibliography}{1}

    \bibitem{WRA}
    J.~Brox, X.~Garc{\'\i}a-Mart{\'\i}nez, M.~Mancini, T.~Van~der Linden, and C.~Vienne, \emph{Weak representability of actions of non-associative algebras}, Journal of Algebra \textbf{669} (2025), 401--444.

    \bibitem{XabiMancini}
    X.~Garc{\'\i}a-Mart{\'\i}nez and M.~Mancini, \emph{Action accessible and weakly action representable varieties of algebras}, Bollettino dell'Unione Matematica Italiana \textbf{19} (2026), no.~2, 367--374.

    \bibitem{Edinburgh}
    X.~Garc{\'\i}a-Mart{\'\i}nez, M.~Tsishyn, T.~Van~der Linden, and C.~Vienne, \emph{Algebras with representable representations}, Proceedings of the Edinburgh Mathematical Society \textbf{64} (2021), no.~3, 555--573.

    \bibitem{WAR}
    G.~Janelidze, \emph{Central extensions of associative algebras and weakly action representable categories}, Theory and Applications of Categories \textbf{38} (2022), no.~36, 1395--1408.

    \bibitem{KISS_MARKI_PROHLE_THOLEN:1982}
    E.~W. Kiss, L.~M\'{a}rki, P.~Pr\"{o}hle, and W.~Tholen, \emph{Categorical algebraic properties. {A} compendium on amalgamation, congruence extension, epimorphisms, residual smallness, and injectivity}, Studia Scientiarum Mathematicarum Hungarica. Combinatorics, Geometry and Topology (CoGeTo) \textbf{18} (1983), no.~1, 79--141.

    \bibitem{SurveyActions}
    M.~Mancini, \emph{Actions in varieties of non-associative algebras: a survey}, submitted, preprint \texttt{arXiv:2610.18302}, 2026.

    \bibitem{MaRe}
    M.~Markl and E.~Remm, \emph{Algebras with one operation including {P}oisson and other {L}ie-admissible algebras}, Journal of Algebra \textbf{299} (2006), no.~1, 171--189.

\end{thebibliography}
% \providecommand{\noopsort}[1]{}
% \providecommand{\bysame}{\leavevmode\hbox to3em{\hrulefill}\thinspace}
% \providecommand{\MR}{\relax\ifhmode\unskip\space\fi MR }
% % \MRhref is called by the amsart/book/proc definition of \MR.
% \providecommand{\MRhref}[2]{%
%   \href{http://www.ams.org/mathscinet-getitem?mr=#1}{#2}
% }
% \providecommand{\href}[2]{#2}

\end{document}

\title[Weak Representability of Actions of Non-Associative Algebras]{Weak Representability of Actions\\of Non-Associative Algebras}

\author{J.~Brox}
\author{X.~García-Martínez}
\author{M.~Mancini}
\author{T.~Van der Linden}
\author{C.~Vienne}

\email{josebrox@uva.es}
\email{xabier.garcia.martinez@uvigo.gal}
\email{manuel.mancini@unipa.it; manuel.mancini@uclouvain.be}
\email{tim.vanderlinden@uclouvain.be; tim.van.der.linden@vub.be}
\email{corentin.vienne@uclouvain.be}

\address[Jose Brox]{Departamento de Álgebra, Análisis Matemático, Geometría y Topología, Universidad de Valladolid, Palacio de Santa Cruz, E--47002 Valladolid, Spain.}
\address[Xabier García-Martínez]{CITMAga \& Universidade de Vigo, Departamento de Ma\-temáticas, Esc.\ Sup.\ de Enx.\ Informática, Campus de Ourense, E--32004 Ourense, Spain.}
\address[Manuel Mancini]{Dipartimento di Matematica e Informatica, Universit\`a degli Studi di Palermo, Via Archirafi 34, 90123 Palermo, Italy.}
\address[Manuel Mancini, Tim Van der Linden, Corentin Vienne]{Institut de Recherche en Mathématique et Physique, Université catholique de Louvain, chemin du cyclotron 2 bte L7.01.02, B--1348 Louvain-la-Neuve, Belgium.}
\address[Tim Van~der Linden]{Mathematics and Data Science, Vrije Universiteit Brussel, Pleinlaan 2, B--1050 Brussel, Belgium}

\thanks{The first author is supported by a postdoctoral fellowship ``Convocatoria 2021'' funded by Universidad de Valladolid, and partially supported by grant PID2022-137283NB-C22 funded by MCIN/AEI/10.13039/501100011033 and ERDF ``A way of making Europe''. The second author is supported by Ministerio de Economía y Competitividad (Spain) with grant number PID2021-127075NA-I00. The third author is supported by the University of Palermo, by the ``National Group for Algebraic and Geometric Structures, and their Applications'' (GNSAGA -- INdAM), by the National Recovery and Resilience Plan (NRRP), Mission 4, Component 2, Investment 1.1, Call for tender No.~1409 published on 14/09/2022 by the Italian Ministry of University and Research (MUR), funded by the European Union -- NextGenerationEU -- Project Title Quantum Models for Logic, Computation and Natural Processes (QM4NP) -- CUP B53D23030160001 -- Grant Assignment Decree No.~1371 adopted on 01/09/2023 by the Italian Ministry of Ministry of University and Research (MUR), by the SDF Sustainability Decision Framework Research Project -- MISE decree of 31/12/2021 (MIMIT Dipartimento per le politiche per le imprese -- Direzione generale per gli incentivi alle imprese) -- CUP:~B79J23000530005, COR:~14019279, Lead Partner:~TD Group Italia Srl, Partner:~Università degli Studi di Palermo, and he is a postdoctoral researcher of the Fonds de la Recherche Scientifique--FNRS. The fourth author is a Senior Research Associate of the Fonds de la Recherche Scientifique--FNRS. The fifth author is supported by the Fonds Thelam of the Fondation Roi Baudouin.}

\begin{abstract}
We study the categorical-algebraic condition that \emph{internal actions are weakly representable} (WRA) in the context of varieties of (non-associative) algebras over a field.

Our first aim is to give a complete characterization of action accessible, operadic quadratic varieties of non-associative algebras which satisfy an identity of degree two and to study the representability of actions for them. Here we prove that the varieties of two-step nilpotent (anti-)commutative algebras and that of commutative associative algebras are weakly action representable, and we explain that the condition (WRA) is closely connected to the existence of a so-called \emph{amalgam}.

Our second aim is to work towards the construction, still within the context of algebras over a field, of a weakly representing object $\E(X)$ for the actions on (or split extensions of) an object $X$. We actually obtain a \emph{partial} algebra $\E(X)$, which we call \emph{external weak actor} of $X$, together with a monomorphism of functors ${\SplExt(-,X) \rightarrowtail \Hom(U(-),\E(X))}$, which we study in detail in the case of quadratic varieties. Furthermore, the relations between the construction of the \emph{universal strict general actor} $\USGA(X)$ and that of $\mathcal{E}(X)$ are described in detail. We end with some open questions.
\end{abstract}

\subjclass[2020]{08A35; 08C05; 16B50; 16W25; 17A32; 17A36; 18C05; 18E13}
\keywords{Action representable category, amalgamation property, split extension, non-associative algebra, partial algebra, quadratic operad}

\maketitle

\section*{Introduction}\label{Introduction}
In the article~\cite{IntAct}, F.~Borceux, G.~Janelidze and G.~M.~Kelly introduce the concept of an \emph{internal object action}, with the aim of extending the correspondence between actions and split extensions from the context of groups and Lie algebras to arbitrary semi-abelian categories~\cite{Semi-Ab}. In certain of those categories, internal actions are exceptionally well behaved, in the sense that the actions on each object~$X$ are \emph{representable}: this means that there exists an object $[X]$, called the \emph{actor} of $X$, such that the functor $\Act(-,X)\cong \SplExt(-,X)$, which sends an object $B$ to the set of actions/split extensions of $B$ on/by~$X$, is naturally isomorphic to the functor $\Hom(-,[X])$. The context of action representable semi-abelian categories is further studied in~\cite{BJK2}, where it is for instance explained that the category of commutative associative algebras over a field is not action representable. Later it was shown that the only action representable variety of non-associative algebras over an infinite field $\F$ of characteristic different from $2$ is the variety of Lie algebras~\cite{Tim}. The relative strength of the notion naturally led to the definition of closely related weaker notions.

The first of these was the concept of an \emph{action accessible} category due to D.~Bourn and G.~Janelidze~\cite{act_accessible}: it is weak enough to include all \emph{Orzech categories of interest}~\cite{Orzech}, as proved by A.~Montoli in~\cite{Montoli}.

Alternatively, the properties of the representing object $[X]$ may be weakened; this is the aim in \cite{Casas}, where it is shown that each Orzech category of interest admits a so-called \emph{universal strict general actor} (USGA for short).

Our present article focuses on a concept which was more recently introduced, by G.~Janelidze in~\cite{WAR}: \emph{weak representability of actions} (WRA). Instead of asking that for each object $X$ in a semi-abelian category $\C$ we have an object $[X]$ and a natural isomorphism $\SplExt(-,X)\cong\Hom_\C(-,[X])$, we require the existence of an object $T$ and a monomorphism of functors
\[
	\tau\colon\SplExt(-,X)\rightarrowtail\Hom_\C(-,T).
\]
Such an object $T$ is then called a \emph{weak actor} of $X$, and when each $X$ admits a weak actor, $\C$ is said to be \emph{weakly action representable}. For instance, if in an Orzech category of interest, each $\USGA(X)$ is an object of the category, then this category is weakly action representable~\cite{CigoliManciniMetere}. This is the case of the category $\AAlg$ of associative algebras~\cite{WAR} or the category $\LeibAlg$ of Leibniz algebras~\cite{CigoliManciniMetere} over a field.

J.~R.~A.~Gray observed in \cite{Gray} that an Orzech category of interest need not be weakly action representable. One of our aims in the present article is to study the condition (WRA) in the context of varieties of (non-associative) algebras over a field. (We recall basic definitions and results concerning this setting in \Cref{SecPrel}.) It is known that such a variety is action accessible if and only if it is algebraically coherent~\cite{Tim}, and it is also known~\cite{WAR} that action accessibility is implied by (WRA).
In~\Cref{SecComm} we give a complete classification of the action accessible, operadic quadratic varieties of non-associative algebras with an identity of degree $2$ (so commutative or anti-commutative algebras) and we study the representability of actions of each of them. Moreover, we prove that the variety of commutative associative algebras, the variety of two-step nilpotent commutative algebras and that of two-step nilpotent anti-commutative algebras are weakly action representable categories. For the variety of commutative associative algebras, we show that the existence of a weak representation is closely connected to the \emph{amalgamation property} (AP) \cite{kiss} which already appeared in~\cite{BJK2} in relation to action representability. In Section~\ref{Section General Representability} the study of (WRA) and its relations with the condition (AP) is extended to a general variety of algebras over a field.

Our second aim is to work towards the construction, still within the context of algebras over a field, of a weakly representing object~$\E(X)$ for the actions on/split extensions of an object $X$ of a variety of non-associative algebras $\V$. We believe that in certain settings, this object may be easier to work with than the more abstract weak actor. In \Cref{E(X)} we actually obtain a \emph{partial} algebra $\E(X)$, which we call \emph{external weak actor} of $X$, together with a monomorphism of functors ${\SplExt(-,X) \rightarrowtail \Hom(U(-),\E(X))}$, where $U$ is the forgetful functor from $\V$ to the category of partial algebras, which we study in detail in the case of quadratic varieties of algebras (\Cref{Section Quadratic}).

We end the article with some open questions (\Cref{Section Questions}).

\section{Preliminaries}\label{SecPrel}

The present work takes place in \emph{semi-abelian categories} which were introduced in \cite{Semi-Ab} in order to capture categorical-algebraic properties of non-abelian algebraic structures. A category is \emph{semi-abelian} if it is pointed, admits binary coproducts, is protomodular and Barr-exact. Well-known examples are the category $\Gp$ of groups, the category $\Rng$ of not necessarily unitary rings, any variety $\V$ of non-associative algebras over a field $\F$, as well as all abelian categories. Throughout the remainder of the paper, when we consider a category $\C$, we assume it to be semi-abelian; when we consider a variety $\V$, we assume that $\V$ is a variety of non-associative algebra over a field $\F$. We fix the field $\F$, so that we may drop it from our notation.

\subsection*{Internal actions and their representability}\label{SecAct}

A central notion which appears in the semi-abelian context is that of \emph{split extensions}. Let $X$, $B$ be objects of a semi-abelian category $\C$; a \emph{split extension} of $B$ by $X$ is a diagram
\begin{equation}\label{eq:split_ext}
	\begin{tikzcd}
		0\ar[r]
		&X\arrow [r, "k"]
		&A \arrow[r, shift left, "\alpha"] &
		B\ar[r]\ar[l, shift left, "\beta"]
		&0
	\end{tikzcd}
\end{equation}
in $\C$ such that $\alpha \circ \beta = 1_B$ and $(X,k)$ is a kernel of $\alpha$. We observe that since protomodularity implies that the pair $(k,\beta)$ is jointly strongly epic, the morphism $\alpha$ is indeed the cokernel of $k$ and diagram (\ref{eq:split_ext}) represents an extension of $B$ by~$X$ in the usual sense. Morphisms of split extensions are morphisms of extensions that commute  with the sections. Let us observe that, again by protomodularity, a morphism of split extensions fixing $X$ and $B$ is necessarily an isomorphism. For an object $X$ of~$\C$, we define the functor
\[
	\SplExt_\C(-,X)\colon \C^{\op} \rightarrow \Set
\]
which assigns to any object $B$ of $\C$, the set $\SplExt_\C(B,X)$ of isomorphism classes of split extensions of $B$ by $X$ in $\C$, and to any arrow $f\colon B'\to B$ the \emph{change of base} function $f^*\colon \SplExt_\C(B,X) \to \SplExt_\C(B',X)$ given by pulling back along $f$. When there is no ambiguity on the category $\C$, we will use the notation $\SplExt(-,X)$.

A feature of semi-abelian categories is that one can define a notion of \emph{internal action}. Internal actions correspond to split extensions via a semidirect product construction; it turns out that, as a result, for our purposes we need no explicit description of what is an internal action. We refer the interested reader to \cite{BJK2}, where the equivalence between the two concepts is described in detail. For us here, it suffices to note that if we fix an object $X$, internal actions on $X$ in $\C$ give rise to a functor
\[
	\Act(-,X)=\Act_\C(-,X) \colon \C^{\op} \rightarrow \Set
\]
and a natural isomorphism of functors $\Act(-,X)\cong \SplExt(-,X)$.  This justifies the terminology in the definition that follows.

\begin{definition}[\cite{BJK2}]
	A semi-abelian category~$\C$ is said to be \defn{action representable} if for every object $X$ in it the functor $\Act(-,X)$ is representable. In other words, there exists an object $[ X]$ in $\C$, called the \defn{actor} of $X$, and a natural isomorphism
	\begin{align*}
		\SplExt(-,X) \cong \Hom_\C(-,[X]).
	\end{align*}
\end{definition}

Basic examples of semi-abelian categories which satisfy action representability are the category $\Gp$ of groups with the actor of $X$ being the group of automorphisms $\Aut(X)$, the category~$\LieAlg$ of Lie algebras with the actor of $X$ being the Lie algebra of derivations $\Der(X)$, and any abelian category with the actor of $X$ being the zero object. For the categories~$\AAlg$ of associative algebras and $\CAAlg$ of commutative associative algebras, representability of actions was studied in~\cite{BJK2}, where the authors proved that they are not action representable.

It is explained in \cite{IntAct} that action representability is equivalent to the condition that for every object $X$ in $\C$ the category $\SplExtX$ of split extensions in $\C$ with kernel $X$ has a terminal object
\[
	\xymatrix{0 \ar[r] & X \ar[r] & [X] \ltimes X \ar@<.5ex>[r] & [X] \ar@<.5ex>[l] \ar[r] & 0.}
\]
We can weaken this condition assuming instead that for any $X$, every object in~$\SplExtX$ is accessible (i.e.\ it has a morphism into a subterminal or so-called \emph{faithful} object, see \cite{act_accessible}). In this way, we encompass a wider class of examples that did not satisfy representability of actions such as the category $\PoisAlg$ of (non-commutative) Poisson algebras, the category $\AAlg$ of associative algebras or the category $\CAAlg$ of commutative associative algebras. This notion called \defn{action accessibility} was introduced by D.~Bourn and G.~Janelidze \cite{act_accessible} in order to calculate centralisers of normal subobjects or of equivalence relations. It was then shown by A.~Montoli that any Orzech category of interest is an action accessible category~\cite{Montoli}. This explains why all of the varieties of non-associative algebras mentioned above are action accessible.

Since by definition the existence of a terminal object in $\SplExtX$ is stronger than every object being accessible, it is immediate that
\[
	\textit{action representability}\Rightarrow\textit{action accessibility.}
\]
Recently, in \cite{WAR}, G.~Janelidze introduced an intermediate notion: \emph{weak representability of actions}.

\begin{definition}
	A semi-abelian category $\C$ is said to be \defn{weakly action representable (WRA)} if for every object $X$ in it, there exists an object $T$ of $\C$ and a monomorphism of functors
	\[
		\tau\colon\SplExt(-,X)\rightarrowtail\Hom_\C(-,T).
	\]
	We call such an object $T$ a \defn{weak actor} of $X$, and a morphism $\varphi \colon B \rightarrow T$ in the image of $\tau_B$ an \defn{acting morphism}.
\end{definition}

It is clear from the definitions that every action representable category is weakly action representable. Also in~\cite{WAR}, it is proven that the category $\AAlg$ is weakly action representable with a weak actor of $X$ given by the associative algebra
\begin{multline*}
	\Bim(X)=\lbrace (f*-,-*f) \in \End(X)\times \End(X)^{\op}\mid f*(xy)=(f*x)y,\\ (xy)*f=x(y*f),\; x(f*y)=(x*f)y, \; \forall x,y \in X \rbrace
\end{multline*}
of \defn{bimultipliers} of $X$ (see~\cite{MacLane58}). The case of the category $\LeibAlg$ of Leibniz algebras was studied in \cite{CigoliManciniMetere}. There the authors showed that a weak actor of a Leibniz algebra~$X$ is the Leibniz algebra
\begin{multline*}
	\Bider(X)=\lbrace (d,D) \in \End(X)^2\mid d(xy)=d(x)y+xd(y), \\
	D(xy)=D(x)y-D(y)x, \; xd(y)=xD(y), \; \forall x,y \in X \rbrace
\end{multline*}
of \defn{biderivations} of $X$ (see~\cite{loday1993version} and \cite{ManciniBider}), where the bilinear operation is defined by
\[
	[(d,D),(d',D')]=(d \circ d' - d' \circ d, D \circ d' - d' \circ D).
\]
In the same paper, the representability of actions in the categories~$\PoisAlg$ and $\CPoisAlg$ of (commutative) Poisson algebras was studied.

Another important observation made by G.~Janelidze is that every weakly action representable category is action accessible. We thus have that
\[
	\textit{action representability}\Rightarrow\textit{weak action representability}\Rightarrow\textit{action accessibility.}
\]
J.~R.~A.\ Gray proved in \cite{Gray} that the varieties of $n$-solvable groups where $n > 3$ are action accessible but not weakly action representable. This partially answers a question asked by G.\ Janelidze in~\cite{WAR}, whether reasonably mild conditions may be found on a semi-abelian category under which the second implication may be reversed: already it makes clear that not all action accessible semi-abelian varieties are weakly action representable. Our aim here is to study what happens for a different class of categories, namely varieties of not necessarily associative algebras over a field.

\subsection*{Varieties of non-associative algebras}\label{SecVar}
We now recall the algebraic setting we are working in: \emph{varieties of non-associative algebras} over a field $\F$. We think of those as collections of algebras satisfying a chosen set of polynomial equations. The interested reader can find a more detailed presentation of the subject in \cite{VdL-NAA}.

By a \defn{(non-associative) algebra} $A$ we mean a vector space $A$ equipped with a bilinear operation $A \times A\rightarrow A\colon (x,y)\mapsto xy$ which we call the \defn{multiplication}. The existence of a unit element is not assumed, nor are any other conditions on the multiplication besides its bilinearity. Let $\NAlg$ denote the category of non-associative algebras, where morphisms are linear maps which preserve the multiplication.

We consider the \emph{free algebra} functor $\Set \rightarrow \NAlg$ which sends a set $S$ to the \emph{free algebra} generated by elements of $S$. This functor has the forgetful functor as a right adjoint. Moreover, it factorises through the \emph{free magma functor} $\FM\colon\Set \rightarrow \Mag$, which sends a set $S$ to the magma $\FM(S)$ of non-associative words in $S$, and the \emph{magma algebra functor} $\F [-] \colon\Mag \rightarrow \NAlg$.

Let $S$ be a set. An element $\varphi$ of $\F[\FM(S)]$ is called a \defn{non-associative polynomial} on $S$. We say that such a polynomial is a \defn{monomial} when it is a scalar multiple of an element in $\FM(S)$. For example, if $S=\lbrace x,y,z,t \rbrace$, then $(xy)t+(zy)x$, $xx+yz$ and $(xt)(yz)$ are polynomials in $S$ and only the last one is a monomial. For a monomial~$\varphi$ on a set $\lbrace x_1, \dots, x_n \rbrace$, we define its \defn{type} as the $n$-tuple $(k_1, \dots , k_n)\in \N^n$ where $k_i$ is the number of times $x_i$ appears in $\varphi$, and its \defn{degree} as the natural number $k_1 + \dots + k_n$. A polynomial is said to be \defn{multilinear} if all monomials composing it have the same type of the form $(1, \dots , 1)$. Among the examples above, only the last one is multilinear.

\begin{definition}\label{def identity variety}
	An \defn{identity} of an algebra $A$ is a non-associative polynomial $\varphi=\varphi(x_1,\dots ,x_n)$ such that $\varphi(a_1, \dots , a_n)=0$ for all $a_1$, \dots, $a_n \in A$. We say that the algebra $A$ \emph{satisfies} the identity $\varphi$.

	Let $I$ be a subset of $\F[\FM(S)]$ with $S$ being a set of variables. The \defn{variety of algebras} determined by $I$ is the class of all algebras which satisfy all the identities in $I$. We say that a variety \defn{satisfies the identities in $I$} if all algebras in this variety satisfy the given identities. In particular, if the variety is determined by a set of multilinear polynomials, then we say that the variety is \defn{operadic}. If there exists a set of identities of degree~$2$ or~$3$ that generate all the identities of~$\V$, we say that the variety is \defn{quadratic}. Recall---see for instance~\cite{DGM} where this is explained in detail---that an operadic, quadratic variety of algebras over a field can be viewed as a variety determined by a quadratic operad.
\end{definition}

Any variety of non-associative algebras can, of course, be seen as a category where the morphisms are the same as in $\NAlg$. In particular, any such variety is a semi-abelian category.

\begin{remark}
	Whenever the characteristic of the field $\F$ is zero, any variety of non-associative algebras over $\F$ is operadic. This is due to the well-known multilinearisation process, see \cite[Corollary~3.7]{Osborn}. The reason behind the name ``operadic'' is explained in~\cite[Section 2]{Cosmash}.
\end{remark}

\begin{examples}\label{Examples varieties}%{\ }
	\begin{enumerate}
		\item We write $\AbAlg$ for the variety of \emph{abelian} algebras determined by the identity $xy=0$. Seen as a category, this variety is isomorphic to the category $\Vect$ of vector spaces over $\F$. It is the only non-trivial variety of non-associative algebras which is an abelian category; this explains the terminology.

		\item We write $\AAlg$ for the variety of \emph{associative} algebras determined by the identity of \emph{associativity} which is $x(yz)-(xy)z=0$, or equivalently $x(yz)=(xy)z$.

		\item We write $\AntiAAlg$ for the variety of \emph{anti-associative} algebras, determined by the \emph{anti-associative} identity $x(yz)=-(xy)z$.

		\item We write $\CAlg$ for the variety of \emph{commutative} algebras determined by the identity of \emph{commutativity} which is $xy-yx=0$, or equivalently $xy=yx$.

		\item We write $\ACAlg$ for the variety of \emph{anti-commutative} algebras determined by \emph{anti-commutativity} which is $xy+yx=0$, or equivalently $xy=-yx$.

		\item We write $\CAAlg$ for the variety of commutative associative algebras.

		\item We write $\ACAntiAAlg$ for the variety of anti-commutative anti-associative algebras.

		\item We write $\LieAlg$ for the variety of \emph{Lie algebras} determined by \emph{anti-commut\-ativity} and the \emph{Jacobi identity}, which respectively are $xy+yx=0$ and $x(yz)+y(zx)+z(xy)=0$.

		\item One can see that all the previous examples are operadic varieties. Let us provide a non-operadic example: the variety $\Bool$ of \emph{Boolean rings}, which may be seen as associative $\Z_2$-algebras satisfying $xx=x$. This variety is action representable.

		\item We write $\JJAlg$ for the variety of \emph{Jacobi--Jordan algebras} which is determined by commutativity and the Jacobi identity. Jacobi--Jordan algebras, also known as \emph{mock-Lie} algebras, are the commutative counterpart of Lie algebras. The name of Jordan in the definition is justified by the fact that every Jacobi--Jordan algebra is a Jordan algebra (see \cite{JJAlg}).

		\item We write $\LeibAlg$ for the variety of \emph{(right) Leibniz algebras} determined by the \emph{(right) Leibniz identity} which is $(xy)z-(xz)y-x(yz)=0$.

		\item We write $\Alt$ for the variety of \emph{alternative algebras}, which is determined by the identities $(yx)x-yx^2=0$ and $x(xy)-x^2y=0$ . Every associative algebra is obviously alternative and an example of an alternative algebra which is not associative is given by the \emph{octonions} $\mathbb{O}$, that is the eight-dimensional algebra with basis $\lbrace e_1,e_2,e_3,e_4,e_5,e_6,e_7,e_8 \rbrace$ and multiplication table
		      \[
			      e_i e_j  =
			      \begin{cases}
				      e_j ,                                      & \text{if }i = 1   \\
				      e_i ,                                      & \text{if }j = 1   \\
				      - \delta_{ij}e_1 + \varepsilon _{ijk} e_k, & \text{otherwise},
			      \end{cases}
		      \]
		      where $\delta_{ij}$ is the \emph{Kronecker delta} and $\varepsilon _{ijk}$ a \emph{completely antisymmetric tensor} with value 1 when $ijk = 123, 145, 176, 246, 257, 347, 365$. Notice that $e_1$ is the unit of the algebra $\mathbb{O}$.

		      When $\bchar(\F) \neq 2$, the multilinearisation process shows that $\Alt$ is equivalent to the variety defined by
		      \[
			      (xy)z+(xz)y-x(yz)-x(zy)=0
		      \]
		      and
		      \[
			      (xy)z+(yx)z-x(yz)-y(xz)=0.
		      \]
		\item Taking any variety $\V$, one can look at a subvariety of it by adding further identities to be satisfied. For example, let $\V$ be a variety determined by a set of identities $I$ and let $k$ be any positive natural number, then we write~$\Nil_k(\V)$ for the variety of \emph{$k$-step nilpotent algebras in $\V$} determined by the identities in $I$ and the identities of the form $x_1 \cdots x_{k+1}=0$ with all possible choices of parentheses.
	\end{enumerate}
\end{examples}

We now want to explain how we may describe actions in a variety of non-associative algebras. As we already mentioned before, in a semi-abelian category, actions are split extensions.
\begin{definition}\label{Rem:SplitV}
	Let
	\begin{equation}\label{diag:SplV}
		\begin{tikzcd}
			0\ar[r]
			&X \arrow [r, "i"]
			&A \arrow[r, shift left, "\pi"] &
			B \ar[r]\ar[l, shift left, "s"]
			&0
		\end{tikzcd}
	\end{equation}
	be a split extension in the variety $\cV$. The pair of bilinear maps
	\[
		l\colon B \times X \rightarrow X, \qquad r \colon X \times B \rightarrow X
	\]
	defined by
	\[
		b \ast x=s(b)i(x), \quad x \ast b=i(x)s(b), \quad \forall b \in B, \; \forall x \in X
	\]
	where $b \ast -=l(b,-)$ and $-\ast b=r(-,b)$, is called the \emph{derived action} of $B$ on $X$ associated with~\eqref{diag:SplV}.
\end{definition}

Given a pair of bilinear maps
\[
	l\colon B \times X \rightarrow X,\qquad r \colon X \times B \rightarrow X
\]
with $B$, $X$ objects of $\V$, we may define a multiplication on the direct sum of vector spaces $B \oplus X$ by
\begin{equation}\label{product on semi direct}
	(b,x) \cdot (b',x')=(bb',xx'+b \ast x' + x \ast b')
\end{equation}
with $ b \ast x' \coloneqq l(b,x')$ and $x\ast b' \coloneqq r(x,b')$. This construction allows us to build the split extension in $\NAlg$
\begin{equation}
	\begin{tikzcd}
		0\ar[r]
		&X \arrow [r, "i_2"]
		& B\oplus X \arrow[r, shift left, "\pi_1"] &
		B \ar[r]\ar[l, shift left, "i_1"]
		&0
	\end{tikzcd}
\end{equation}
with $i_2(x)=(0,x)$, $i_1(b)=(b,0)$ and $\pi_1(b,x)=b$. This is a split extension in $\V$ if and only if $(B\oplus X,\cdot)$ is an object of $\V$, i.e.\ it satisfies the identities which determine~$\V$. In other words, we have the following result analogous to~\cite[Theorem 2.4]{Orzech} and~\cite[Lemma 1.8]{Tim}:

\begin{lemma}\label{lemma derived action}
	In a variety of non-associative algebras $\V$, given a pair of bilinear maps
	\[
		l\colon B \times X \rightarrow X,\qquad r \colon X \times B \rightarrow X,
	\]
	we define the multiplication on $B \oplus X$ as above in~\eqref{product on semi direct}. Then, the pair $(l,r)$ is a derived action of $B$ on $X$ if and only if $(B \oplus X,\cdot)$ is in $\V$. In this case, we call~$B\oplus X$ the \defn{semi-direct product} of $B$ and $X$ (with respect to the derived action) and we denote it by $B\ltimes X$.
\end{lemma}

\begin{remark}\label{rem_b*}
	Notice that, for any split extension~\eqref{diag:SplV} and the corresponding derived action $(l,r)$, there is an isomorphism of split extensions
	\begin{equation*}
		\begin{tikzcd}
			0\ar[r]
			&X \arrow [r, "i_2"] \ar[d, "1_{X}"']
			&B \ltimes X \arrow[r, shift left, "\pi_1"] \ar[d, "\theta"]&
			B \ar[r]\ar[l, shift left, "i_1"]\ar[d, "1_{B}"]
			&0 \\
			0\ar[r]
			&X \arrow [r, "i"]
			&A \arrow[r, shift left, "\pi"] &
			B \ar[r]\ar[l, shift left, "s"]
			&0
		\end{tikzcd}
	\end{equation*}
	where $\theta \colon B \ltimes X \rightarrow A\colon (b,x)\mapsto s(b)+i(x)$. 	Thus, when we write $b \ast x$ (resp.\ $x \ast b$), one can think of it as the multiplication $(b,0)\cdot (0,x)$ (resp.\ $(0,x)\cdot(b,0$)) in $B \ltimes X$.
\end{remark}

\subsection*{Categorical consequences}

Let $\V$ be an operadic variety of non-associative algebras. We recall two results which will be useful for understanding the rest of the paper.

\begin{theorem}[\cite{GM-VdL2,GM-VdL3}]\label{Theorem AC iff Orzech}
	The following conditions are equivalent:
	\begin{tfae}
		\item $\V$ is \defn{algebraically coherent}~\cite{acc};
		\item $\V$ is an Orzech category of interest;
		\item $\V$ is action accessible;
		\item there exist $\lambda_{1}$, \dots, $\lambda_{8}$, $\mu_{1}$, \dots, $\mu_{8}$ in $\F$ such that
		\begin{align*}
			x(yz)=\lambda_1(xy)z & +\lambda_2(yx)z+\lambda_3z(xy) + \lambda_4 z(yx)                        \\
			                     & + \lambda_5 (xz)y + \lambda_6 (zx)y + \lambda_7 y(xz) + \lambda_8 y(zx)
			\intertext{and}
			(yz)x=\mu_1(xy)z     & +\mu_2(yx)z+\mu_3z(xy) + \mu_4 z(yx)                                    \\
			                     & + \mu_5 (xz)y + \mu_6 (zx)y + \mu_7 y(xz) + \mu_8 y(zx)
		\end{align*}
		are identities in $\V$. \noproof
	\end{tfae}
\end{theorem}

We call the two previous identities together the \defn{$\lambda/\mu$-rules}. Since (WRA) implies action accessibility in general, the existence of the $\lambda/\mu$-rules is a necessary condition for the variety $\V$ to be weakly action representable.

\begin{theorem}[\cite{Tim}]\label{AR implies Lie}
	The following conditions are equivalent:
	\begin{tfae}
		\item $\V$ is action representable;
		\item $\V$ is either the variety $\LieAlg$ or the variety $\AbAlg$.\noproof
	\end{tfae}
\end{theorem}

\Cref{AR implies Lie} helps motivating our interest in the condition (WRA). In fact, in our context, there is only one non-trivial example of a variety which is action representable. This suggests to study a generalisation of the notion of representability of actions. On the other hand, action accessibility may not be enough to study some kind of (weak) actor. The next result, which is closely related to \cite[Proposition 4.5]{WAR}, explains one way of understanding weak action representability for any variety of non-associative algebras over a field.

\begin{proposition}
	A variety of non-associative algebras $\V$ is weakly action representable if and only if for any object $X$ in it, there exists an object $T$ of $\V$ such that for every derived action of an object $B$ of $\V$ on $X$
	\[
		l\colon B \times X \rightarrow X, \qquad r \colon X \times B \rightarrow X,
	\]
	there exists a unique morphism $\varphi \in \Hom_\V(B,T)$ and a derived action $(l',r')$ of~$\varphi(B)$ on $X$ such that
	\[
		l'(\varphi(b),x)= l(b,x), \quad r'(x,\varphi(b))= r(x,b),
	\]
	for every $b \in B$ and for every $x \in X$.
\end{proposition}

\begin{proof}
	($\Rightarrow$) If $\V$ is weakly action representable, then for any object $X$ in it there exists a weak representation $(T,\tau)$. Let $B$ be an object of $\V$ which acts on $X$ and let $\varphi \colon B \rightarrow T$ be the corresponding acting morphism. Consider the split extension diagram
	\begin{equation*}
		\begin{tikzcd}
			0\ar[r]
			&X \arrow [r, "i"] \ar[d, "1_{X}"']
			&B \ltimes X \arrow[r, shift left, "\pi"] \ar[d, dashrightarrow, "\exists ! f"]&
			B \ar[r]\ar[l, shift left, "s"]\ar[d, "\widetilde{\varphi}"]
			&0 \\
			0\ar[r]
			&X \arrow [r, "i'"]
			&\varphi(B) \ltimes X \arrow[r, shift left, "\pi'"] &
			\varphi(B) \ar[r]\ar[l, shift left, "s'"]
			&0
		\end{tikzcd}
	\end{equation*}
	where $\widetilde{\varphi}$ is the corestriction of $\varphi$ to its image, $i'(x)=(0,x)$, $s'(\varphi(b))=(\varphi(c),0)$, where $(c,0)=s(b)$, and $f(b,x)=(\varphi(b),x)$. Then the action of $\varphi(B)$ on $X$ is defined by the pair of bilinear maps
	\[
		l'\colon \varphi(B) \times X \rightarrow X, \qquad r' \colon X \times \varphi(B) \rightarrow X
	\]
	where
	\[
		l'(\varphi(b),x)=s'(\varphi(b)) i'(x) = s(b) i(x) = l(b,x)
	\]
	and
	\[
		r'(\varphi(b),x)=i(x) s'(\varphi(b)) = i(x) s(b) = r(b,x),
	\]
	for every $b \in B$ and for every $x \in X$ (we use \cite[Proposition 4.5]{WAR} to see that $l'$ and $r'$ are well defined).

	($\Leftarrow$) Conversely, given an object $X$ of $\V$, a weak representation of $\SplExt(-,X)$ is given by $(T,\tau)$, where the component
	\[
		\tau_B \colon \SplExt(B,X) \rightarrowtail \Hom_\V(B,T)
	\]
	sends every action of $B$ on $X$ to the corresponding morphism $\varphi$. Moreover, $\tau_B$ is an injection since the morphism $\varphi$ is uniquely determined by the action of $B$ on $X$. Thus $\tau$ is a monomorphism of functors.
\end{proof}

\subsection*{Partial Algebras}

We end this chapter with a notion we shall use throughout the text.

\begin{definition}
	Let $X$ be an $\F$-vector space. A \emph{bilinear partial operation} on $X$ is a map
	\[
		\cdot \colon \Omega \rightarrow X,
	\]
	where $\Omega$ is a vector subspace of $X \times X$, which is bilinear on $\Omega$, i.e.\
	\[
		(\alpha_1 x_1 + \alpha_2 x_2) \cdot y = \alpha_1 x_1 \cdot y + \alpha_2 x_2 \cdot y
	\]
	for any  $\alpha_1, \alpha_2 \in \F$ and $x_1, x_2, y \in X$ such that $(x_1,y),(x_2,y) \in \Omega$ and
	\[
		x \cdot (\beta_1 y_1 + \beta_2 y_2) = \beta_1 x \cdot y_1 + \beta_2 x \cdot y_2
	\]
	for any $\beta_1, \beta_2 \in \F$ and $x, y_1, y_2 \in X$ such that $(x,y_1),(x,y_2) \in \Omega$.
\end{definition}

\begin{definition}
	A \emph{partial algebra} over $\F$ is an $\F$-vector space $X$ endowed with a bilinear partial operation
	\[
		\cdot \colon \Omega \rightarrow X.
	\]
	We denote it by $(X,\cdot, \Omega)$. When $\Omega = X \times X$ we say that the algebra is \emph{total}.
\end{definition}

Let $(X,\cdot,\Omega)$ and $(X',\ast,\Omega')$ be partial algebras over $\F$. A homomorphism of partial algebras is an $\F$-linear map $f\colon X \rightarrow X'$ such that $f(x\cdot y)=f(x)\ast f(y)$ whenever $(x,y)\in \Omega$, which tacitly implies that $(f(x),f(y))\in \Omega'$ (i.e.\ both $x\cdot y$ and $f(x)\ast f(y)$ are defined). We denote by $\PAlg$ the category whose objects are partial algebras and whose morphisms are partial algebra homomorphisms.

\begin{definition}
	We say that a partial algebra $(X, \cdot, \Omega)$ \emph{satisfies an identity} when that identity holds wherever the bilinear partial operation is well defined.
\end{definition}

For instance, a partial algebra $(X, \cdot, \Omega)$ is associative if
\[
	x \cdot (y \cdot z)=(x \cdot y) \cdot z
\]
for every $x$, $y$, $z \in X$ such that $(x,y)$, $(y,z)$, $(x,yz)$, $(xy,z) \in \Omega$.

\begin{remark}
	We observe that any variety of non-associative algebras $\cV$ has an obvious forgetful functor $U \colon \cV \rightarrow \PAlg$.
\end{remark}

\section{Commutative and anti-commutative algebras}\label{SecComm}

In this section we aim to study the (weak) representability of actions of some varieties of non-associative algebras which satisfy the commutative law or the anti-commutative law. As explained in Section \ref{SecPrel}, we may assume our variety satisfies the $\lambda/\mu$-rules, or equivalently is action accessible.

When $\mathcal{V}$ is either a variety of commutative or anti-commutative algebras, i.e.\ $xy=\varepsilon yx$ is an identity of $\cV$, with $\varepsilon = \pm 1$, the $\lambda/\mu$-rules reduce to
\[
	x(yz)=\alpha(xy)z+\beta(xz)y,
\]
for some $\alpha$, $\beta \in \F$. The following proposition is a representation theory exercise:

\begin{proposition}\label{prop1}
	Let $\cV$ be a non-abelian, action accessible, operadic variety of non-associative algebras.
	\begin{enumerate}
		\item If $\cV$ is a variety of commutative algebras, then $\cV$ is a either a subvariety of $\CAAlg$ or a subvariety of $\JJAlg$.
		\item If $\cV$ is a variety of anti-commutative algebras, then $\cV$ is either a subvariety of $\LieAlg$ or a subvariety of $\ACAntiAAlg$.\noproof
	\end{enumerate}
\end{proposition}

\begin{remark}
	We observe that $\Nil_2(\CAlg)$ is a subvariety of both $\CAAlg$ and $\JJAlg$: in fact, from $x(yz)=(xy)z=0$ we may deduce that associativity holds and the Jacobi identity is satisfied:
	\[
		x(yz)+y(zx)+z(xy)=0+0+0=0.
	\]
	If $\bchar(\F) \neq 3$, then $\Nil_2(\CAlg)$ is precisely the intersection of the varieties $\CAAlg$ and $\JJAlg$. Indeed, let $\cV$ be a subvariety of both $\CAAlg$ and $\JJAlg$. Since commutativity, associativity and the Jacobi identity hold in $\cV$, we have
	\[
		(xy)z=x(yz)=-y(zx)-z(xy)=-x(yz)-(xy)z=-2(xy)z
	\]
	and thus $3(xy)z=3x(yz)=0$.

	An example of an algebra which lies in the intersection of $\CAAlg$ and $\JJAlg$ but which is not two-step nilpotent is the two-dimensional $\F_3$-algebra with basis $\lbrace e_1,e_2 \rbrace$ and bilinear multiplication determined by
	\[
		e_1^2=e_1e_2=e_2e_1=e_2^2=e_2.
	\]

	Likewise, $\Nil_2(\ACAlg)$ is a subvariety of both $\LieAlg$ and $\ACAntiAAlg$: from $x(yz)=(xy)z=0$ we may deduce anti-associativity and the Jacobi identity. If $\bchar(\F) \neq 3$, then $\Nil_2(\ACAlg)$ coincides with the intersection of the varieties $\LieAlg$ and $\ACAntiAAlg$. Indeed, let $\cV$ be a subvariety of both $\LieAlg$ and $\ACAntiAAlg$. Since anti-commutativity, anti-associativity and the Jacobi identity hold in $\cV$, we have
	\[
		(xy)z=-x(yz)=-(xy)z-y(xz)=-(xy)z+(yx)z =-2(xy)z
	\]
	and thus $3(xy)z=-3x(yz)=0$.

	When $\bchar(\F)=3$, it is possible to construct an algebra that lies in the intersection of $\LieAlg$ and $\ACAntiAAlg$ but which is not two-step nilpotent. Let $X$ be the algebra of dimension $7$ over $\F_3$ with basis
	\[
		\lbrace e_1,e_2,e_3,e_4,e_5,e_6,e_7 \rbrace
	\]
	and bilinear multiplication determined by
	\[
		e_1 e_2=-e_2 e_1=e_4, \; e_1 e_3=-e_3 e_1=-e_6, \; e_2 e_3=-e_3 e_2=e_5
	\]
	and
	\[
		e_1 e_5=-e_5 e_1 = e_2 e_6 =- e_6 e_2=e_3,e_4=-e_4 e_3=e_7.
	\]
	Then $X$ is a Lie algebra such that
	\[
		x(xy)=0
	\]
	for any $x$, $y \in X$ and, using the multi-linearisation process, one can check this identity is equivalent to anti-associativity if the characteristic of the field is different from $2$. This $X$ is not two-step nilpotent, since
	\[
		e_1(e_2 e_3)=e_1 e_5 =e_7.
	\]
\end{remark}

\begin{corollary}
	Let $\cV$ be an action accessible, operadic, quadratic variety of non-associative algebras and suppose that $\cV$ is not the variety $\AbAlg$ of abelian algebras.
	\begin{enumerate}
		\item[(1)] If $\cV$ is commutative, then it has to be one of the following varieties:  $\JJAlg$, $\CAAlg$, their intersection, or $\Nil_2(\CAlg)$.
		\item[(2)] If $\cV$ is anti-commutative, then it has to be one of the following varieties: $\LieAlg$, $\ACAntiAAlg$, their intersection, or $\Nil_2(\ACAlg)$.	\noproof
	\end{enumerate}
\end{corollary}

We already know that~$\LieAlg$ is action representable and that the actor of a Lie algebra $X$ is the Lie algebra $\Der(X)$ of derivations of~$X$. Therefore, we shall study the representability of actions of the varieties $\CAAlg$, $\JJAlg$, $\Nil_2(\CAlg)$, $\ACAntiAAlg$ and $\Nil_2(\ACAlg)$.

\subsection*{Commutative associative algebras}

The representability of actions of the variety of commutative associative algebras over a field was studied in~\cite{BJK2}, where F.~Borceux, G.~Janelidze and G.~M.~Kelly proved that it is not action representable. We want to extend this result proving that the variety $\CAAlg$ is weakly action representable. In Section~\ref{Section General Representability} this is further extended to general algebras over a field. We start by recalling the following result, where $U\colon \CAAlg\to \AAlg$ denotes the forgetful functor.

\begin{lemma}[\cite{BJK2}, proof of Theorem~2.6]\label{RemCAAlg}
	Let $X$ be a commutative associative algebra. There exists a natural isomorphism of functors from $\CAAlg^{\op}\to\Set$ which we denote
	\[
		\rho \colon \SplExt(-,X) \cong \Hom_{\AAlg}(U(-),\M(X)),
	\]
	where $\SplExt(-,X)=\SplExt_\CAAlg(-,X)$ and
	\[
		\M(X)=\lbrace f \in \End(X) \mid f(xy)=f(x)y, \quad \forall x,y \in X \rbrace
	\]
	is the associative algebra of \emph{multipliers} of $X$, endowed with the product induced by the usual composition of functions (see~\cite{Casas, MacLane58}).\noproof
\end{lemma}

We recall that $\M(X)$ in general does not need to be a commutative algebra. For instance, let $X=\F^2$ be the abelian two-dimensional algebra, then~$\M(X) = \End(X)$ which is not commutative. However there are special cases where $\M(X)$ is an object of $\CAAlg$, such as when the \emph{annihilator} of~$X$ (which coincides with the categorical notion of center)
\[
	\Ann(X)=\lbrace x \in X \mid xy=0, \; \forall y \in X \rbrace
\]
is trivial or when $X^2=X$, where $X^2$ denotes the subalgebra of $X$ generated by the products $xy$ where $x$, $y\in X$. We refer the reader to~\cite{Casas} for further details.

\begin{theorem}[\cite{BJK2}, Theorem~2.6]\label{characterisation1}
	Let $X$ be a commutative associative algebra. The following statements are equivalent:
	\begin{tfae}
		\item $\M(X)$ is a commutative associative algebra;
		\item the functor $\SplExt(-,X)$ is representable.\noproof
	\end{tfae}
\end{theorem}

Since we have examples where $\M(X)$ is not commutative, we conclude that $\CAAlg$ is not action representable. We now want to prove that it is a weakly action representable category. We analyse what this means and then prove that the category does indeed fulfil these requirements.

For any commutative associative algebra $T$, the fully faithful embedding $U$ of the category $\CAAlg$ into $\AAlg$ induces a natural isomorphism
\[
	i\colon \Hom_{\CAAlg}(-,T) \cong \Hom_{\AAlg}(U(-),U(T))\colon \CAAlg^{\op}\to \Set\text{.}
\]

\begin{lemma}
	If the functor $\SplExt(-,X)$ admits a weak representation $(T,\tau)$, then there exists an injective function $j\colon \M(X)\to T$ such that for each commutative associative algebra $B$, the square
	\[
		\xymatrix{\SplExt(B,X) \ar[d]_-{\rho_B} \ar[r]^-{\tau_B} & \Hom_{\CAAlg}(B,T) \ar[d]_-{\cong}^{i_B}  \\
		\Hom_{\AAlg}(U(B),\M(X)) \ar[r]_-{j\circ(-)} & \Hom_{\AAlg}(U(B),U(T))
		}
	\]
	commutes.
\end{lemma}
\begin{proof}
	The free associative $\F$-algebra on a single generator is the algebra $\F[x]$ of non-constant polynomials in a single variable $x$, which since it is commutative is also the free algebra on a single generator in $\CAAlg$. We find $j$ as the injective function
	\[
		V(\M(X))\cong \Hom_{\AAlg}(\F[x],\M(X))\to \Hom_{\AAlg}(\F[x],U(T))\cong V(U(T))
	\]
	where $V(A)$ denotes the underlying set of an algebra $A$ and the function in the middle is the $\F[x]$-component of the monomorphism of functors
	\[
		i \circ \tau \circ \rho^{-1}\colon \Hom_{\AAlg}(U(-),\M(X)) \rightarrowtail \Hom_{\AAlg}(U(-),U(T)).
	\]
	Now each $b\in B$ induces a morphism $b\colon \F[x]\to B$, and the collection of morphisms ${(b\colon \F[x]\to B)_{b\in B}}$ is jointly epic. Hence its image
	\[
		(\Hom_{\AAlg}(U(B),U(T))\to \Hom_{\AAlg}(U(\F[x]),U(T)))_{b\in B}
	\]
	through the contravariant functor $\Hom_{\AAlg}(U(-),U(T))$ is a jointly monic collection of arrows. It thus suffices that for each $b\in B$, the outer rectangle in the diagram
	\[
		\xymatrix@C=5em{\SplExt(B,X) \ar[d]_-{\rho_B} \ar[r]^-{\tau_B} & \Hom_{\CAAlg}(B,T) \ar[d]_-{\cong}^{i_B}  \\
		\Hom_{\AAlg}(U(B),\M(X)) \ar[d]_-{(-)\circ U(b)} \ar[r]_-{j\circ(-)} & \Hom_{\AAlg}(U(B),U(T)) \ar[d]^-{(-)\circ U(b)} \\
		\Hom_{\AAlg}(U(\F[x]),\M(X)) \ar[d]_\cong \ar[r]_-{i_{\F[x]} \circ \tau_{\F[x]} \circ \rho^{-1}_{\F[x]}} & \Hom_{\AAlg}(U(\F[x]),U(T)) \ar[d]^\cong\\
		V(\M(X)) \ar[r]_-{j} & V(U(T))
		}
	\]
	commutes in $\Set$. This is an immediate consequence of the naturality of the transformations involved.
\end{proof}

\begin{remark}
	The above proof can be modified to show that the function $j$ is in fact a vector space monomorphism. If it were moreover an algebra monomorphism, then this would yield a proof that all $\M(X)$ are commutative, which is false by the above-mentioned example. Thus we would be able to conclude that $\CAAlg$ is not weakly action representable. Theorem~\ref{Theorem CAAlg WRA} below proves that this is wrong. Hence~$j$ cannot preserve the algebra multiplication in general.
\end{remark}

Each action $\xi$ of a commutative associative algebra $B$ on $X$ gives rise to a morphism $\rho_B(\xi)\colon U(B)\to \M(X)$ in $\AAlg$. If the actions in $\CAAlg$ are weakly representable, then $\xi$ is also determined by a morphism of commutative associative algebras $\tau_B(\xi)\colon B\to T$. The above lemma tells us that $j\circ \rho_B(\xi)=\tau_B(\xi)$. Note that here we drop the $i_B$ for the sake of clarity.

Each $\rho_B(\xi)\colon U(B)\to \M(X)$ is the composite of a surjective commutative associative algebra map $\rho'_B(\xi)\colon U(B)\to M_\xi$ and the canonical inclusion of a subalgebra $M_\xi$ of $\M(X)$. We find a diagram of subalgebras of $\M(X)$ indexed over the commutative associative algebra actions on $X$. Note that since trivial actions exist, the image in $\AAlg$ of the diagram $(M_\xi)_\xi$ actually consists of all commutative subalgebras of $\M(X)$, with the canonical inclusions between them. We may re-index and view $(M_\xi)_\xi$ as a diagram in $\AAlg$ over the thin category of commutative subalgebras of $\M(X)$.

By the above, the $M_\xi$ further include into $T$ via $j$. For each $\xi$, an image factorisation of $\tau_B(\xi)\colon B\to T$ is given by the surjective algebra map $\rho'_B(\xi)\colon U(B)\to M_\xi$ followed by the inclusion of $M_\xi$ into $\M(X)$ composed with~$j$. We denote this function $\mu_{M_\xi}\colon M_\xi\to T$ and note that it only depends on the object~$M_\xi$. (That is to say, if $\xi$ and $\psi$ are two $B$-actions such that $M_\xi=M_\psi$, then the induced inclusions into $T$ coincide as well.) \emph{A~priori} this $\mu_{M_\xi}$ is only an injective \emph{map}, but since $\tau_B(\xi)$ and $\rho'_B(\xi)$ are morphisms of algebras and $\rho'_B(\xi)$ is a surjection, that injection is a monomorphism of commutative associative algebras. Furthermore, the $\mu_{M_\xi}$ form a cocone on the diagram of all commutative subalgebras of $\M(X)$ with vertex $T$.

Recall that for a diagram in a category, an \emph{amalgam} is a monic cocone, i.e.\ a cocone which is a monomorphic natural transformation. This means that each component of that cocone is a monomorphism, which implies that all the morphisms of the given diagram were monomorphisms to begin with. Note that in a category with colimits, an amalgam for a diagram exists if and only if its colimit cocone is such an amalgam. A category is said to have the \emph{amalgamation property} (AP) when each span of monomorphisms admits an amalgam; equivalently, for each pushout square
\[
	\xymatrix{I \ar[r]^t \ar[d]_s & T\ar[d]^-{\iota_T}\\
	S \ar[r]_-{\iota_S} & S+_IT}
\]
if $s$ and $t$ are monomorphisms then so are $\iota_S$ and $\iota_T$. It is known that neither the category of associative algebras, nor the category of commutative associative algebras satisfies the condition (AP)---see~\cite{kiss} for an overview of examples and references to the rich literature on the subject.

This is as follows related to the problem at hand. The associative algebra $\M(X)$ is an amalgam in $\AAlg$ of the diagram consisting of the commutative subalgebras $M_\xi$ of $\M(X)$. So if the functor $\SplExt(-,X)$ admits a weak representation $(T,\tau)$, then the natural transformation $\tau$ factors through the diagram $(M_\xi)_\xi$ as explained above, and we see that $(T,\tau)$ restricts to an amalgam of that diagram in the category $\CAAlg$.

Thus we find a necessary condition for weak representability of actions in the category $\CAAlg$: we need that for each commutative associative algebra $X$, the diagram $(M_\xi)_\xi$ of commutative subalgebras of the associative algebra $\M(X)$ not only admits the amalgam $\M(X)$ in the category $\AAlg$; it should also admit an amalgam $T$ in~$\CAAlg$. Actually, the converse also holds:

\begin{proposition}\label{Proposition Amalgam}
	For a commutative associative algebra $X$, a weak representation $(T,\tau)$ of $\SplExt(-,X)$ exists if and only if an amalgam in $\CAAlg$ exists for the diagram of commutative subalgebras of $\M(X)$.
\end{proposition}
\begin{proof}
	We already explained that any weak representation of $\SplExt(-,X)$ restricts to such an amalgam. So let us assume that a commutative amalgam for the diagram of commutative subalgebras of $\M(X)$ exists. For each commutative associative algebra action $\xi$ of an object~$B$ on $X$, we let $\tau_B(\xi)\colon B\to T$ be the composite of $\rho'_B(\xi)\colon U(B)\to M_\xi$ with the inclusion $\mu_{M_\xi}\colon M_\xi\to T$ of $M_\xi$ into the amalgam $T$.

	The thus defined $\tau$ is a natural transformation by the naturality of both $\rho'$ and the cocone components in the amalgam. Note that if two maps, say $\rho_B(\xi)$ and $\rho_C(\psi)$, to $\M(X)$ have the same image subalgebra $M_\xi=M_\psi$ of $\M(X)$, then by naturality of $\rho$ and the fact that the inclusion of $M_\xi$ into $\M(X)$ is a monomorphism, for any equivariant map $f\colon B\to C$ we have that the square on the left
	\[
		\xymatrix{B \ar[d]_-{f} \ar[r]^-{\rho'_B(\xi)} & M_\xi \ar@{=}[d] \ar[r]^-{\mu_{M_\xi}} & T \ar@{=}[d]\\
		C \ar[r]_-{\rho'_C(\psi)}  & M_\psi \ar[r]_-{\mu_{M_\psi}} & T}
	\]
	commutes. The commutativity of the entire diagram proves naturality of $\tau$.

	We still have to prove that the components of $\tau$ are monomorphisms: two different actions $\xi$ and $\psi$ of $B$ on $X$ give rise to two different maps $\tau_B(\xi), \tau_B(\psi)\colon B\to T$. Suppose, on the contrary, that $\tau_B(\xi)=\tau_B(\psi)$. Then by uniqueness of image factorisations, the images of $\mu_{M_\xi}\colon M_\xi\to T$ and  $\mu_{M_\psi}\colon M_\psi\to T$ are isomorphic subobjects of $T$. Now the image in $\AAlg$ of the diagram $(M_\xi)_\xi$ is a thin category, so that $M_\xi=M_\psi$. Hence $\mu_{M_\xi}\circ \rho'_B(\xi)=\tau_B(\xi)=\tau_B(\psi)=\mu_{M_\psi}\circ\rho'_B(\psi)=\mu_{M_\xi}\circ\rho'_B(\psi)$, which implies $\rho'_B(\xi)=\rho'_B(\psi)$. But then the actions $\xi$ and $\psi$ are equal, since $\rho$ is a natural isomorphism by Lemma~\ref{RemCAAlg}.
\end{proof}

Thus we see that the problem of weak representability of actions of $\CAAlg$ amounts to proving that an amalgam in $\CAAlg$ exists for the diagram of commutative subalgebras of $\M(X)$ for any object $X$. We are actually going to prove something a bit stronger: namely, that an amalgam in $\CAAlg$ exists for any diagram of commutative associative algebras for which an amalgam exists in $\AAlg$. The essence of the proof is contained in the following special case.

\begin{theorem}\label{Thm Amalgam}
	If $S\ot I \to T$ is a span of commutative associative algebras for which an amalgam exists in $\AAlg$, then it has an amalgam in $\CAAlg$.
\end{theorem}

The proof depends on the following lemma.

\begin{lemma}\label{Lemma Mono}
	Let $f\colon X\to Y$ and $g\colon Y\to Z$ be morphisms in a semi-abelian category. The composite $g\circ f\colon X\to Z$ is a monomorphism if and only if $f$ is a monomorphism and $\Imm(f)\cap \Ker(g)$ is trivial.
\end{lemma}
\begin{proof}
	Note that if $g\circ f$ is a monomorphism, then so is $f$. So we may assume that $f$ is a monomorphism in either case. The composite $g \circ f$ is a monomorphism precisely when $\Ker(g \circ f)$ is trivial. Now this kernel is a pullback of $\Ker(g)$ along~$f$. Since $f$ is a monomorphism, that pullback is $\Imm(f)\cap \Ker(g)$. So $\Ker(g\circ f)$ is zero if and only if $\Imm(f)\cap \Ker(g)$ is zero.
\end{proof}

\begin{proof}[Proof of Theorem~\ref{Thm Amalgam}]
	Let $S\ot I \to T$ be such a span. Recall that an amalgam in either category exists if and only if the $S$ and $T$ components of the induced pushout cocone in either category are monic. We focus on the case $\iota_S\colon S\to S+_IT=S+^{\AAlg}_IT$ which we assume to be monic. The question is whether its composite with the reflection unit $\eta_{S+_IT}\colon S+^{\AAlg}_IT\to S+^{\CAAlg}_IT$ is still monic. We are going to prove that the answer is yes, indeed it is.

	Consider the following morphism of short exact sequences in $\AAlg$:
	\[
		\xymatrix{0 \ar[r] & J \ar[d]_{\eta_{S+T}|_J} \ar[d] \ar[r] & S+T \ar[d]_{\eta_{S+T}} \ar[r] & S+_IT \ar[d]^{\eta_{S+_IT}} \ar[r] & 0\\
		0 \ar[r] & K \ar[r] & S+^{\CAAlg}T \ar[r] & S+^{\CAAlg}_IT \ar[r] & 0}
	\]
	As a vector space, the coproduct $S+T$ of $S$ and $T$ in $\AAlg$ is $S\oplus T\oplus U$ where $U=(S\tensor T)\oplus (T\tensor S)\oplus (S\tensor T\tensor S)\oplus \cdots$ contains all the tensors. The coproduct $S+^{\CAAlg}T$ of $S$ and $T$ in $\CAAlg$ is $S\oplus T\oplus (S\tensor T)$, so that $\eta_{S+T}$ admits a canonical splitting $\sigma\colon S+^{\CAAlg}T\to S+T$ in $\Vect$ which commutes with the inclusions of $S$, $T$ and $S\tensor T$.

	We note that $K$ is the ideal of $S+^{\CAAlg}T=S\oplus T\oplus (S\tensor T)$ generated by the elements of the form $\underline{i}-\overline{i}$, where $\underline{i}$ is $i\in I$ viewed as an element of~$S$, while $\overline{i}$ is $i$ viewed as an element of $T$. Let $G$ denote the set of generators $\{\underline{i}-\overline{i}\mid i\in I\}$. The algebra $J$ is generated by $G$ as well, but now as an ideal of $S+T=S\oplus T\oplus U$. It follows that $\sigma(K)\subseteq J$---even though $\sigma$ is not a morphism of algebras. We give a detailed proof of this claim. We know that $K$ consists of all elements of the form $x_1g_1+\cdots +x_ng_n$ where $x_1$, \dots, $x_n\in S+^{\CAAlg}T$ and $g_1$, \dots, $g_n\in G$. Since $\sigma$ is a morphism of abelian groups, it suffices that $\sigma(xg)$ belongs to $J$ for all $x\in S+^{\CAAlg}T$ and $g\in G$. Now each $x\in S+^{\CAAlg}T$ is of the form $(s_1,t_1,s'_1\tensor t'_1)+\cdots+(s_n,t_n,s'_n\tensor t'_n)$ with $s_1$, \dots, $s_n$, $s'_1$, \dots, $s'_n\in S$ and $t_1$, \dots, $t_n$, $t'_1$, \dots, $t'_n\in T$. Hence it suffices to prove that  $\sigma((s,t,s'\tensor t')g)$ belongs to $J$ for all $s$, $s'\in S$, $t$, $t'\in T$ and $g\in G$. Next, we see that $(s,t,s'\tensor t')=(s,0,0)+(0,t,0)+(0,0,s'\tensor t')$. As a consequence, it suffices to prove that $\sigma((s,0,0)g)$, $\sigma((0,t,0)g)$ and $\sigma((0,0,s\tensor t)g)$ belong to $J$ for all $s\in S$ and $t\in T$. Writing $g=(i,-i,0)\in S+^{\CAAlg}T=S\oplus T\oplus (S\tensor T)$ and $S+T=S\oplus T\oplus U$, we calculate what happens in each of these three cases:
	\begin{align*}
		\sigma((s,0,0)g) & =\sigma((s,0,0)(i,-i,0))=\sigma(si,0,-s\tensor i)=(si,0,-s\tensor i) \\&=(s,0,0)(i,-i,0)\in J,
	\end{align*}
	\begin{align*}
		\sigma((0,t,0)g) & =\sigma((0,t,0)(i,-i,0))=\sigma(0,-it,i\tensor t)=(0,-it,i\tensor t) \\&=(i,-i,0)(0,t,0)\in J,
	\end{align*}
	\begin{align*}
		\sigma((0,0,s\tensor t)g) & =\sigma((0,0,s\tensor t)(i,-i,0))=\sigma(0,0,si\tensor t-s\tensor it) \\&=(0,0,si\tensor t-s\tensor it) =(si,0,-s\tensor i)(0,t,0)\\&=(s,0,0)(i,-i,0)(0,t,0)\in J.
	\end{align*}
	Note that in the above calculations, we used commutativity twice: in the second equality of the second and third cases.

	Thanks to Lemma~\ref{Lemma Mono}, inside the vector space $S+T$, the intersection $S\cap J$ is trivial, by the assumption that the composite ${\iota_S\colon S\to S+T\to S+_IT}$ is monic. But then the smaller space $S\cap \sigma(K)$ is trivial as well, so that the composite
	\[
		S\to S+^{\CAAlg}T\to S+^{\CAAlg}_IT
	\]
	is a monomorphism by Lemma~\ref{Lemma Mono}.
\end{proof}

For arbitrary diagrams of monomorphisms of commutative associative algebras, the proof stays essentially the same. This allows us to conclude:

\begin{theorem}\label{Theorem CAAlg WRA}
	The category $\CAAlg$ of commutative associative algebras is weakly action representable.\noproof
\end{theorem}

\begin{remark}
	For a given diagram of commutative associative algebras as above, the amalgam $T$ in $\CAAlg$ is also an amalgam in $\AAlg$.
\end{remark}

\begin{remark}\label{reminitial}
	Note that by its construction as a colimit, the weak representation $(T,\tau)$ is automatically an \emph{initial weak representation} (see \cite[Section 5]{WAR}). As explained in \cite[Corollary~5.3, Corollary 5.4]{WAR}, the existence of the initial weak representation also follows from the existence of a weak representation and the fact that $\CAAlg$, as a semi-abelian variety of \emph{universal} algebras, is a total category~\cite{Street}.
\end{remark}

\subsection*{Jacobi--Jordan algebras}

We now want to study the representability of actions of the variety $\JJAlg$ of Jacobi--Jordan algebras. As already mentioned in Section~\ref{SecPrel}, every split extension of $B$ by~$X$ in $\LieAlg$ is represented by a homomorphism ${B \rightarrow \Der(X)}$. For Jacobi--Jordan algebras, the role the derivations have in $\LieAlg$ is played by the so-called \emph{anti-derivations}.

\begin{definition}
	Let $X$ be a Jacobi--Jordan algebra. An \emph{anti-derivation} is a linear map $d \colon X \rightarrow X$ such that
	\[
		d(xy)=-d(x)y-d(y)x, \quad \forall x,y \in X.
	\]
\end{definition}

The (left) multiplications $L_x$ for $x \in X$ are particular anti-derivations, called \emph{inner anti-derivations}. We denote by $\ADer(X)$ the space of anti-derivations of $X$ and by $\Inn(X)$ the subspace of the inner anti-derivations. Anti-derivations play a significant role in the study of cohomology of Jacobi--Jordan algebras: see~\cite{JJAlg2} for further details.

We now want to make explicit what are the derived actions in the category~$\JJAlg$ and how they are related with the anti-derivations. The following is an easy application of \Cref{lemma derived action}.

\begin{proposition}
	Let $X$ and $B$ be two Jacobi--Jordan algebras. Given a pair of bilinear maps
	\[
		l\colon B \times X \rightarrow X, \qquad r \colon X \times B \rightarrow X
	\]
	defined by
	\[
		b \ast x = l(b,x), \qquad x \ast b = r(x,b),
	\]
	we construct $(B \oplus X,\cdot)$ as in \eqref{product on semi direct}. Then $(B \oplus X, \cdot)$ is a Jacobi--Jordan algebra if and only if
	\begin{enumerate}
		\item $b\ast x =x \ast b$;
		\item $b \ast (xx')=-(b \ast x)x'-(b \ast x) \ast x'$;
		\item $(bb') \ast x=-b \ast (b' \ast x) - b' \ast (b \ast x)$;
	\end{enumerate}
	for all $b$, $b' \in B$ and $x$, $x' \in X$.\noproof
\end{proposition}

In an equivalent way, a derived action of $B$ on $X$ in the variety $\JJAlg$ is given by a linear map
\[
	B \rightarrow \ADer(X)\colon b \mapsto b\ast-
\]
which satisfies
\begin{equation}\label{JJeq}
	(bb') \ast x = - b\ast (b' \ast x) - b' \ast (b \ast x), \quad \forall b,b' \in B, \; \forall x \in X.
\end{equation}

\begin{remark}\label{remJJ}
	The vector space $\ADer(X)$ endowed with the \emph{anti-commutator}
	\[
		\langle -,- \rangle \colon \ADer(X) \times \ADer(X) \rightarrow \End(X) , \quad \langle f,f' \rangle = -f \circ f'-f' \circ f
	\]
	is not in general an algebra, since the anti-commutator of two anti-derivations is not in general an anti-derivation: in \cite[Remark 2.2]{JJAlg3}, the authors proved that $\langle f,f' \rangle \in \ADer(X)$ if and only if
	\[
		\langle f,f' \rangle (xy)= - f(x)f'(y) - f'(x)f(y), \quad \forall x,y \in X.
	\]
	Moreover, it can happen that the anti-commutator $\langle -,- \rangle$ is a well defined bilinear operation on the space $\ADer(X)$ but it does not define a Jacobi--Jordan algebra structure: for instance, if $X=\F$ is the abelian one-dimensional algebra, then $\ADer(X)=\End(X) \cong \F$ (every linear endomorphism of~$X$ is of the form $\varphi_\alpha \colon x \mapsto \alpha x$, for some $\alpha \in \F$) and the Jacobi identity is not satisfied. Nevertheless, there are some subspaces of $\ADer(X)$ that are Jacobi--Jordan algebras. For instance, the subspace $\Inn(X)$ of all inner anti-derivations of~$X$. Indeed, the linear map
	\[
		X \rightarrow \ADer(X)\colon x \mapsto L_x,
	\]
	restricts to a Jacobi--Jordan algebra homomorphism $X \rightarrow \Inn(X)$. This is true in general for the image of any linear map $B \rightarrow \ADer(X)$ satisfying equation~\eqref{JJeq}.
\end{remark}

Thus we need to use an algebraic structure which includes the space of anti-derivations endowed with the anti-commutator and which allows us to describe categorically the representability of actions of the variety $\JJAlg$. One possible solution is given by \emph{partial algebras}.

Indeed, the vector space $\ADer(X)$ endowed with the anti-commutator $\langle-,-\rangle$ is a commutative partial algebra. In this case $\Omega$ is the preimage
\[
	\langle-,-\rangle^{-1}(\ADer(X))
\]
of the inclusion $\ADer(X)\hookrightarrow \End(X)$.
\begin{theorem}\label{thmJJAlg}
	Let $X$ be a Jacobi--Jordan algebra and let $U \colon \JJAlg^{\op} \rightarrow \PAlg$ denote the forgetful functor.
	\begin{enumerate}
		\item There exists a natural isomorphism of functors from $\JJAlg^{\op} \rightarrow \Set$
		      \[
			      \rho \colon \SplExt(-,X)\cong\Hom_{\PAlg}(U(-),\ADer(X)),
		      \]
		      where $\SplExt(-,X)=\SplExt_\JJAlg(-,X)$;
		\item if $\ADer(X)$ is a Jacobi--Jordan algebra, then the functor $\SplExt(-,X)$ is representable and $\ADer(X)$ is the actor of $X$;
		      %		\item if the functor $\SplExt(-,X)$ admits a weak representation, then $\ADer(X)$ satisfies the Jacobi identity.
	\end{enumerate}
\end{theorem}

\begin{proof}
	(1) For a Jacobi--Jordan algebra $B$, we define the component
	\[
		\rho_{B}\colon \SplExt(B,X) \rightarrow \Hom_{\PAlg}(U(B),\ADer(X))
	\]
	as the functor which sends any split extension
	\begin{equation*}
		\begin{tikzcd}
			0\ar[r]
			&X \arrow [r, "i"]
			&A \arrow[r, shift left, "\pi"] &
			B \ar[r]\ar[l, shift left, "s"]
			&0
		\end{tikzcd}
	\end{equation*}
	to the morphism $B \rightarrow \ADer(X)\colon b \mapsto b \ast -$. The transformation $\rho$ is natural. Indeed, for any Jacobi--Jordan algebra homomorphism $f\colon B' \rightarrow B$, it is easy to check that the diagram in $\Set$
	\[
		\begin{tikzcd}
			\SplExt(B,X) \arrow[r, "\rho_{B}"] \arrow[d, "{\SplExt(f,X)}"]
			& \Hom(U(B),\ADer(X)) \arrow[d,"{\Hom(U(f),\ADer(X))}"]\\
			\SplExt(B',X) \arrow[r, "\rho_{B'}"] & \Hom(U(B'),\ADer(X))
		\end{tikzcd}
	\]
	where $\Hom(U(-),-)=\Hom_{\PAlg}(U(-),-)$, is commutative. Moreover, for any Jacobi--Jordan algebra $B$, the morphism $\rho_{B}$ is an injection, as each element of $\SplExt(B,X)$ is uniquely determined by the corresponding action of $B$ on $X$. Thus $\rho$ is a monomorphism of functors. Finally $\rho$ is a natural isomorphism since, given any Jacobi--Jordan algebra $B$ and any homomorphism of partial algebras $\varphi \colon B \rightarrow \ADer(X)$, the bilinear maps $l_\varphi \colon B \times X \rightarrow X\colon (b,x)\mapsto\varphi(b)(x)$, $r_\varphi=l_\varphi$ define a (unique) derived action of $B$ on $X$ such that $\rho_B(l_\varphi,r_\varphi)=\varphi$.

	(2) If $\ADer(X)$ is a Jacobi--Jordan algebra, then by (1) we have a natural isomorphism
	\[
		\SplExt(-,X)\cong\Hom_{\JJAlg}(-,\ADer(X)),
	\]
	hence $\ADer(X)$ is the actor of $X$.
\end{proof}

\subsection*{Two-step nilpotent commutative algebras}

We now analyse the case where~$\cV$ is a subvariety of both $\CAAlg$ and $\JJAlg$, i.e.\ $\cV$ is the variety $\Nil_2(\CAlg)$ of two-step nilpotent commutative algebras. We recall this means that $xyz=0$ is an identity of $\cV$. An example of such an algebra is the \emph{Kronecker algebra} $\mathfrak{k}_1$ (see~\cite{LaRosaMancini1}), which is the three-dimensional algebra with basis $\lbrace e_1,e_2,e_3 \rbrace$ and multiplication determined by $e_1e_2=e_2e_1=e_3$.

We shall show that $\Nil_2(\CAlg)$ is an example of a weakly action representable, operadic, quadratic variety of commutative algebras.

\begin{proposition}\label{propNil2}
	Let $X$ and $B$ be two algebras in $\Nil_2(\CAlg)$. Given a pair of bilinear maps
	\[
		l\colon B \times X \rightarrow X, \qquad r \colon X \times B \rightarrow X,
	\]
	we construct $(B \oplus X,\cdot)$ as in \eqref{product on semi direct}. Then $(B \oplus X, \cdot)$ is in $\Nil_2(\CAlg)$ if and only if
	\begin{enumerate}
		\item $b\ast x =x \ast b$;
		\item $b \ast (xx')=(b \ast x)x'=0$;
		\item $(bb') \ast x=b \ast (b' \ast x) =0$;
	\end{enumerate}
	for any $b$, $b' \in B$ and $x$, $x' \in X$.\noproof
\end{proposition}

The second equation of~\Cref{propNil2} states that, for every $b \in B$, the linear map $b \ast -$ belongs to the vector space
\[
	[X]_2=\lbrace f \in \End(X) \mid f(xy)=f(x)y=0, \; \forall x \in X \rbrace.
\]
Moreover, seeing $[X]_2$ as an abelian algebra (i.e.\ $\langle f, g \rangle =0_{\End(X)}$, for every $f,g \in [X]_2$), from the third equation we deduce that the linear map
\[
	B \rightarrow [X]_2\colon b \mapsto b \ast -
\]
is an algebra homomorphism.

On the other hand, given a morphism of algebras
\[
	\varphi\colon B \rightarrow [X]_2, \quad \varphi(b) = b \ast -
\]
satisfying
\[
	b \ast (b' \ast x) =0, \quad \forall b,b' \in B, \; \forall x \in X,
\]
we can consider the split extension
\begin{equation*}
	\begin{tikzcd}
		0\ar[r]
		&X \arrow [r, "i"]
		&(B \oplus X, \ast_\varphi) \arrow[r, shift left, "\pi"] &
		B \ar[r]\ar[l, shift left, "s"]
		&0
	\end{tikzcd}
\end{equation*}
where the two-step nilpotent commutative algebra structure of $B \oplus X$ is given by
\[
	(b,x)\ast_\varphi(b',x')=(bb',xx'+b \ast x' + b' \ast x), \quad \forall (b,x),(b',x') \in B \oplus X.
\]

We can now claim the following result.

\begin{theorem}\label{Nil_2com}%{\ }
	\begin{enumerate}
		\item Let $B$ and $X$ be two-step nilpotent commutative algebras. The isomorphism classes of split extensions of $B$ by $X$ are in bijection with the algebra homomorphisms
		      \[
			      B \rightarrow [X]_2\colon b \mapsto b \ast -
		      \]
		      satisfying
		      \begin{equation}\label{Nil_2}
			      b \ast (b' \ast x)=0, \quad \forall b,b' \in B,\; \forall x \in X.
		      \end{equation}
		\item The variety $\Nil_2(\CAlg)$ is weakly action representable. A weak representation of the functor $\SplExt(-,X)=\SplExt_{\Nil_2(\CAlg)}(-,X)$ is given by
		      \[
			      \tau \colon \SplExt(-,X) \rightarrowtail \Hom_{\Nil_2(\CAlg)}(-,[X]_2),
		      \]
		      where $\tau_B$ is the injection which sends any split extension of $B$ by~$X$ to the corresponding homomorphism $B \rightarrow [X]_2$, defined by $b \mapsto b \ast -$ as above.
		\item A homomorphism $B \rightarrow [X]_2$ is an acting morphism if and only if it satisfies~\Cref{Nil_2}.
	\end{enumerate}
\end{theorem}

\begin{proof}
	(1) It follows from the analysis above.

	(2) We observe that $\tau$ is a natural transformation. Indeed, for every morphism $f\colon B' \rightarrow B$ in $\Nil_2(\CAlg)$, we can check that the diagram in $\Set$
	\[
		\begin{tikzcd}
			\SplExt(B,X) \arrow[r, "\tau_{B}"] \arrow[d, "{\SplExt(f,X)}"]
			& \Hom(B,[X]_2) \arrow[d,"{\Hom(f,[X]_2)}"]\\
			\SplExt(B',X) \arrow[r, "\tau_{B'}"] & \Hom(B',[X]_2)
		\end{tikzcd}
	\]
	is commutative.	Moreover $\tau_{B}$ is an injection since every isomorphism class of split extensions of $B$ by $X$ is uniquely determined by the corresponding derived action.	Thus $\tau$ is a monomorphism of functors and $\Nil_2(\CAlg)$ is a weakly action representable category.

	(3) Finally, $\varphi \colon B \rightarrow [X]_2$ is an acting morphism if and only if it defines a split extension of $B$ by $X$ in $\Nil_2(\CAlg)$, i.e.\ it satisfies equation~\eqref{Nil_2}.
\end{proof}

Let us observe that not every morphism $B \rightarrow [X]_2$ defines a split extension of~$B$ by $X$. For instance, if $B=\F\{b,b'\}$ and $X=\F\{x\} \cong \F$ are abelian algebras, then~$[X]_2=\End(X)$ and the homomorphism $\varphi \colon B \rightarrow [X]_2$, defined by
\[
	\varphi(b)=\varphi(b')=1_{X}
\]
is not an acting morphism. Indeed,
\[
	\varphi(b)(\varphi(b')(x))=1_X(1_X(x))=x \neq 0.
\]

\subsection*{Anti-commutative anti-associative algebras}

For the variety $\ACAntiAAlg$ of anti-commutative anti-associative algebras, a similar description of split extensions and derived actions can be made as for the variety $\JJAlg$. The role of the anti-derivations is played here by the endomorphisms in the associative partial algebra
\[
	[X] \coloneqq \lbrace f \in \End(X) \; \vert \; f(xy)=-f(x)y, \; \forall x \in X \rbrace,
\]
whose bilinear partial operation is given by
\[
	\langle f,g \rangle = - f \circ g.
\]
It is easy to see that $\langle -,- \rangle$ does not define, in general, a total algebra structure on~$[X]$, nor need it be anti-commutative or anti-associative. An example is given by the abelian two-dimensional algebra $X=\F^2$, where $[X]=\End(X)$.

We may check that a derived action of $B$ by $X$ in the variety $\ACAntiAAlg$ is the same thing as a partial algebra homomorphism
\[
	B \rightarrow [X]\colon b \mapsto b\ast-
\]
which satisfies
\begin{equation*}\label{ACAAA}
	(bb') \ast - = -b \ast (b' \ast -), \quad \forall b,b' \in B.
\end{equation*}
Moreover, we obtain the following result whose proof is similar to the one of \Cref{thmJJAlg}.
\begin{theorem}\label{thmAnti}
	Let $X$ be a an object of $\ACAntiAAlg$ and let $U \colon \ACAntiAAlg \rightarrow \PAlg$ denote the forgetful functor.
	\begin{enumerate}
		\item There exists a natural isomorphism \[\SplExt(-,X)\cong\Hom_{\PAlg}(U(-),[X]),\]
		      where $\SplExt(-,X)=\SplExt_\ACAntiAAlg(-,X)$;
		\item if $[X]$ is an anti-commutative anti-associative algebra, then the functor $\SplExt(-,X)$ is representable and $[X]$ is the actor of $X$; \noproof
	\end{enumerate}
\end{theorem}

\subsection*{Two-step nilpotent anti-commutative algebras}

We conclude this section by studying the representability of actions of the variety $\Nil_2(\ACAlg)$. An important example of a two-step nilpotent anti-commutative algebra is the $(2n+1)$-dimensional \emph{Heisenberg algebra} $\mathfrak{h}_{2n+1}$, that is the algebra with basis
\[
	\lbrace e_1,\ldots,e_n,f_1,\ldots,f_n,h\rbrace
\]
and non-trivial products $e_if_j=-f_je_i=\delta_{ij}h$, for all $i$, $j=1,\ldots,n$, where $\delta_{ij}$ is the \emph{Kronecker delta}.

An analysis, similar to the case of two-step nilpotent commutative algebras can be made, so we simply state the following theorem:

\begin{theorem}%{\ }
	\begin{enumerate}
		\item Let $B$ and $X$ be two-step nilpotent anti-commutative algebras. The isomorphism classes of split extensions of $B$ by $X$ are in bijection with the algebra homomorphisms
		      \[
			      B \rightarrow [X]_2\colon	b \mapsto b \ast -
		      \]
		      where $[X]_2$ is defined as in the commutative case, which satisfy the condition
		      \begin{equation}\label{Nil_2Lie}
			      b \ast (b' \ast x)=0, \quad \forall b,b' \in B,\; \forall x \in X.
		      \end{equation}
		\item The variety $\Nil_2(\ACAlg)$ is weakly action representable. A weak representation of $\SplExt(-,X)=\SplExt_{\Nil_2(\ACAlg)}$ is given by
		      \[
			      \tau \colon \SplExt(-,X) \rightarrowtail \Hom_{\Nil_2(\ACAlg)}(-,[X]_2),
		      \]
		      where $\tau_B$ is the injection which associates with any split extension of $B$ by~$X$, the corresponding homomorphism $B \rightarrow [X]_2\colon b \mapsto b \ast -$ as in (1).
		\item A homomorphism $B \rightarrow [X]_2$ is an acting morphism if and only if it satisfies~\Cref{Nil_2Lie}.\noproof
	\end{enumerate}
\end{theorem}

Again, if $B=\F\{b,b'\}$ is the abelian two-dimensional algebra and $X=\F$ is the abelian one-dimensional algebra, the linear map $\varphi \colon B \rightarrow [X]_2=\End(X)$, defined by $\varphi(b)=\varphi(b')=1_X$ is an example of a morphism in $\Nil_2(\ACAlg)$ which is not an acting morphism.

\section{Representability of actions of non-associative algebras}\label{Section General Representability}

We want to extend the results obtained in the previous section by studying the (weak) representability of actions of a general variety of non-associative algebras over a field~$\F$. Again, we assume that $\mathcal{V}$ is an action accessible, operadic variety of non-associative algebras over~$\F$. Thus $\cV$ satisfies a set of multilinear identities
\[
	\Phi_{k,i}(x_1,\ldots,x_k)=0, \quad i=1,\ldots,n,
\]
where $k$ is the degree of the polynomial $\Phi_{k,i}$. We fix $\lambda_1$, \dots, $\lambda_8$, $\mu_1$, \dots, $\mu_8 \in \F$ which determine a choice of $\lambda/\mu$ rules, i.e.\
\begin{align*}
	x(yz)=\lambda_1(xy)z & +\lambda_2(yx)z+\lambda_3z(xy) + \lambda_4 z(yx)                        \\
	                     & + \lambda_5 (xz)y + \lambda_6 (zx)y + \lambda_7 y(xz) + \lambda_8 y(zx)
	\intertext{and}
	(yz)x=\mu_1(xy)z     & +\mu_2(yx)z+\mu_3z(xy) + \mu_4 z(yx)                                    \\
	                     & + \mu_5 (xz)y + \mu_6 (zx)y + \mu_7 y(xz) + \mu_8 y(zx)
\end{align*}
which are identities in $\mathcal{V}$. Note that these are not unique, but fixed for our purposes.

For any object $X$ of $\cV$, we want to define a vector space $\E(X)$ such that
\[
	\Inn(X) \leq \E(X) \leq \End(X)^2,
\]
where $\Inn(X)=\lbrace (L_x, R_x) \mid x \in X \rbrace$ is the vector space of left and right multiplications of $X$, and we want to endow it with a bilinear partial operation
\[
	\langle -,- \rangle \colon \Omega \subseteq X \times X \rightarrow X,
\]
such that we can associate in a natural way a homomorphism of partial algebras~$B \rightarrow \E(X)$, with every split extension of $B$ by $X$ in $\cV$. To do this, we describe derived actions in $\V$ in a similar fashion as in the previous section.

\begin{proposition}\label{prop_semidirect}
	Let $X$ and $B$ be two algebras in $\V$. Given a pair of bilinear maps
	\[
		l\colon B \times X \rightarrow X, \qquad r \colon X \times B \rightarrow X,
	\]
	we construct $(B \oplus X,\cdot)$ as in \eqref{product on semi direct}. Then $(B \oplus X, \cdot)$ is an object of $\V$ if and only if
	\[
		\Phi_{k,i}(\alpha_1,\ldots,\alpha_k)=0, \quad \forall i=1,\ldots,n,
	\]
	where at least one of the $\alpha_1$, \dots, $\alpha_k$ is an element of of the form $(0,x)$, with~$x \in X$, and the others are of the form $(b,0)$, with $b \in B$. The resulting algebra is the \emph{semi-direct product} of $B$ and $X$, denoted by $B \ltimes X$.\noproof
\end{proposition}

Using the same notation of \Cref{rem_b*}, we obtain the following:

\begin{corollary}
	When every identity of $\cV$ can be deduced from the $\lambda/\mu$ rules, $(B \oplus X, \cdot)$ is an object of $\cV$ if and only if
	\begin{enumerate}
		\item $b \ast (xx')=\lambda_1 (b \ast x)x'+ \cdots + \lambda_8 x (x' \ast b)$;
		\item $(xx') \ast b=\mu_1 (b \ast x)x'+ \cdots + \mu_8 x (x' \ast b)$;
		\item $x(x' \ast b)=\lambda_1 (xx') \ast b + \cdots + \lambda_8 x' (b \ast x)$;
		\item $(x' \ast b)x=\mu_1 (xx') \ast b + \cdots+\mu_8 x' (b \ast x)$;
		\item $x(b \ast x')=\lambda_1 (x \ast b)x' + \cdots +\lambda_8 b \ast (x'x)$;
		\item $(b \ast x')x=\mu_1 (x \ast b)x' + \cdots +\mu_8 b \ast (x'x)$;
		\item $x \ast (bb')=\lambda_1 (x \ast b) \ast b' + \cdots + \lambda_8 b \ast (b' \ast x)$;
		\item $(bb') \ast x=\mu_1 (x \ast b) \ast b' + \cdots + \mu_8 b \ast (b' \ast x)$;
		\item $b \ast (b' \ast x)=\lambda_1 (bb') \ast x + \cdots + \lambda_8 b' \ast (x \ast b)$;
		\item $(b' \ast x) \ast b = \mu_1 (bb') \ast x + \cdots + \mu_8 b' \ast (x \ast b)$;
		\item $b \ast (x \ast b')= \lambda_1 (b \ast x) \ast b' + \cdots + \lambda_8 x \ast (b'b)$;
		\item $(x \ast b) \ast b'= \mu_1 (b \ast x) \ast b' + \cdots + \mu_8 x \ast (b'b)$,	\end{enumerate}
	for all $b$, $b' \in B$ and $x$, $x' \in X$.\noproof
\end{corollary}

\begin{definition}\label{E(X)}
	For every object $X$ of $\mathcal{V}$, we define $\E(X)$ as the subspace of all pairs $(f\ast-,-\ast f) \in \End(X)^2$ satisfying
	\[
		\Phi_{k,i}(\alpha_1,\ldots,\alpha_k)=0, \quad \forall i=1,\ldots,n,
	\]
	for each choice of $\alpha_j=f$ and $\alpha_t \in X$, where $t\neq j\in \{1,\ldots,k\}$ and $fx\coloneq f\ast x$, $xf\coloneq x\ast f$. We endow it with the bilinear map $\langle -,- \rangle \colon \E(X) \times \E(X) \rightarrow \End(X)^2$
	\[
		\langle (f\ast-,-\ast f),(g \ast -, - \ast g) \rangle =(h \ast -,- \ast h),
	\]
	where
	\begin{align*}
		x \ast h=\lambda_1(x\ast f)\ast g & +\lambda_2(f\ast x)\ast g+\lambda_3g\ast(x\ast f) + \lambda_4 g \ast(f\ast x)                                   \\
		                                  & + \lambda_5 (x\ast g)\ast f + \lambda_6 (g\ast x)\ast f + \lambda_7 f\ast (x\ast g) + \lambda_8 f\ast (g\ast x)
		\intertext{and}
		h \ast x=\mu_1(x\ast f)\ast g     & +\mu_2(f\ast x)\ast g+\mu_3g\ast(x\ast f) + \mu_4 g\ast(f\ast x)                                                \\
		                                  & + \mu_5 (x\ast g)\ast f + \mu_6 (g\ast x)\ast f + \mu_7 f\ast (x\ast g) + \mu_8 f\ast (g\ast x).
	\end{align*}
\end{definition}

When every identity of $\cV$ is a consequence of the $\lambda/\mu$ rules, $\E(X)$ becomes the subspace of all pairs $(f\ast-,-\ast f) \in \End(X)^2$ satisfying
\begin{enumerate}
	\item $f \ast (xx')=\lambda_1 (f \ast x)x'+ \cdots + \lambda_8 x (x' \ast f)$;
	\item $(xx') \ast f=\mu_1 (f \ast x)x'+ \cdots + \mu_8 x (x' \ast f)$;
	\item $x(x' \ast f)=\lambda_1 (xx') \ast f + \cdots + \lambda_8 x' (f \ast x)$;
	\item $(x' \ast f)x=\mu_1 (xx') \ast f + \cdots + \mu_8 x' (f \ast x)$;
	\item $x(f \ast x')=\lambda_1 (x \ast f)x' + \cdots + \lambda_8 f \ast (x'x)$;
	\item $(f \ast x')x=\mu_1 (x \ast f)x' + \cdots + \mu_8 f \ast (x'x)$,
\end{enumerate}
for every $x$, $x' \in X$.

Note that the choice of $\lambda/\mu$ rules does not affect to the definition of the underlying vector space of $\E(X)$, but it does play an important role in the bilinear map $\langle -,- \rangle$.
In general, the vector space $\E(X)$ endowed with the bilinear map $\langle -,- \rangle$ is not an object of $\cV$. It may happen that $\langle -,- \rangle$ does not even define a bilinear operation on $\E(X)$, i.e.\ there exist $(f\ast-,-\ast f)$,$(g\ast-,-\ast g) \in \E(X)$ such that
\[
	\langle (f\ast-,-\ast f),(g\ast-,-\ast g)\rangle \not\in \E(X)
\]
or that $(\E(X),\langle -,- \rangle)$ is a non-associative algebra which does not satisfy some identity of $\cV$.

\begin{example}\label{Example associative}
	We may check that, if $\mathcal{V}=\AAlg$, then $\E(X)\cong\Bim(X)$ as vector spaces. Moreover, with the standard choice of $\lambda/\mu$ rules $\lambda_1 = \mu_8 = 1$ and the rest equal to zero, it is also an isomorphism of associative algebras.
\end{example}

\begin{example}
	Let $\mathcal{V}=\LeibAlg$, it is easy to see that $\E(X) \cong \Bider(X)$ as vector spaces. Choosing the $\lambda/\mu$ rules as
	\begin{align*}
		x(yz) & =(xy)z - (xz)y, \\
		(yz)x & =(yx)z - y(xz),
	\end{align*}
	we get the standard multiplication defined in~$\Bider(X)$ as in~\cite{loday1993version}, that defines a weak actor in $\LeibAlg$. On the other hand, choosing the $\lambda/\mu$ rules as
	\begin{align*}
		x(yz) & =(xy)z - (xz)y, \\
		(yz)x & =(yx)z + y(zx),
	\end{align*}
	we get the non-associative algebra structure defined in~\cite[Definition~5.2]{Casas}, which, in general, is not a Leibniz algebra.
\end{example}

\begin{example}\label{Example Nilk}
	If $\mathcal{V}=\Nil_k(\AAlg)$, with $k \geq 3$, then
	\[
		\E(X)=\lbrace (f \ast -, - \ast f) \in \Bim(X) \mid f \ast (x_1 \cdots x_k)=(x_1 \cdots x_k) \ast f =0 \rbrace.
	\]
	With the same choice of $\lambda/\mu$ rules as in~\Cref{Example associative}, the bilinear operation $\langle -,- \rangle$ becomes
	\[
		\langle (f\ast-,-\ast f),(g\ast-,-\ast g) \rangle = (f \ast (g \ast -), (- \ast f) \ast g)
	\]
	which makes $\E(X)$ an associative algebra, but not a $k$-step nilpotent algebra. For instance, let $X$ be the abelian one-dimensional algebra, then
	\[
		\E(X)=\End(X) \times \End(X)^{\op} \cong \F^2
	\]
	which is not nilpotent. Indeed, every linear endomorphism of~$X$ is of the form $\varphi_\alpha \colon x \mapsto \alpha x$, for some $\alpha \in \F$ and
	\[
		\langle (\varphi_\alpha,\varphi_\beta),(\varphi_{\alpha'},\varphi_{\beta'}) \rangle = (\varphi_\alpha \circ \varphi_{\alpha'},\varphi_{\beta'} \circ \varphi_\beta)=(\varphi_{\alpha \alpha'},\varphi_{\beta \beta'}).
	\]
\end{example}

\begin{example}\label{exalt}
	If $\cV=\Alt$ is the variety of alternative algebras over a field $\F$ with $\bchar(\F) \neq 2$, then $\E(X)$ consists of the pairs $(f \ast -, - \ast f)  \in \End(X)^2$ satisfying
	\[
		f \ast (xy) = (x \ast f) y + (f \ast x) y -x (f \ast y),
	\]
	\[
		(xy) \ast f = x (f \ast y) + x (y \ast f) - (x \ast f) y,
	\]
	\[
		x (y \ast f) = (yx) \ast f + (xy) \ast f - y (x \ast f)
	\]
	and
	\[
		(f \ast x) y = f \ast (yx) + f \ast (xy) - (f \ast y) x
	\]
	for any $x,y \in X$, and the bilinear map
	\[
		\langle (f\ast-,-\ast f),(g \ast -, - \ast g) \rangle =(h \ast -,- \ast h)
	\]
	is given by
	\[
		h \ast x=-(f \ast x) \ast g + f \ast (g \ast x) + f \ast (x \ast g)
	\]
	and
	\[
		x \ast h = (x \ast f) \ast g + (f \ast x) \ast g - f \ast (x \ast g).
	\]
	One can check that $\langle -,- \rangle $ does not define an algebra structure. Nevertheless, it is possible to find examples where $\E(X)$ is an alternative algebra.

	For instance, if $X$ is a \emph{unitary}, or \emph{unital}, alternative algebra (i.e.\ there exists an element $e \in X$ such that $xe=ex=x$, for any $x \in X$), such as the algebra of octonions $\mathbb{O}$, then the elements of $\E(X)$ satisfy the following set of equations
	\begin{align*}
		f \ast x = & (x \ast f) e + (f \ast x) e -x (f \ast e),  \\
		x \ast f = & e (f \ast x) + e (x \ast f) - (e \ast f) x, \\
		x \ast f = & x \ast f + x \ast f - x (e \ast f),         \\
		f \ast x = & f \ast x + f \ast x - (f \ast e) x,
	\end{align*}
	for any $x \in X$. Thus, if $\alpha \coloneqq f \ast e$ and $\beta \coloneqq e \ast f$, one has
	\[
		f \ast x = \alpha x = \beta x, \qquad x \ast f = x \alpha = x \beta
	\]
	and, for $x = e$, one obtains $\alpha=\beta$. In other words, an element of $\E(X)$ is uniquely determined by an element $\alpha=f \ast e = e \ast f$ of $X$, i.e.
	\[
		\E(X) \cong \lbrace (\alpha,\alpha)  \mid \alpha \in X \rbrace \cong X
	\]
	is an object of $\Alt$.
\end{example}

\begin{remark}\label{unitalassociative}
	The same result can be obtained for unitary algebras in the variety $\AAlg$. In fact, let $X$ be a unitary associative algebra and let $(f \ast -, - \ast f) \in \Bim(X)$. Thus
	\[
		f \ast x = f \ast (e x) = \alpha x,
	\]
	\[
		x \ast f = (x e) \ast f = x \beta
	\]
	and
	\[
		x \alpha = (x \ast f) e = x \beta,
	\]
	where $\alpha \coloneqq f \ast e$ and $\beta \coloneqq e \ast f$. For $x=e$, we obtain $\alpha=\beta$ and
	\[
		\Bim(X) \cong \lbrace (\alpha,\alpha) \mid \alpha \in X \rbrace \cong X.
	\]
	Since unitary algebras are perfect and have trivial center, from \cite{Casas} we have a natural isomorphism
	\[
		\SplExt(-,X) \cong \Hom_{\AAlg}(-,X)
	\]
	for any unitary associative algebra $X$, i.e.\ $X$ is its own actor.
\end{remark}

The construction of $\E(X)$ gives rise to an alternative characterisation of the split extensions in $\cV$. In fact, a split extension of $B$ by $X$ in $\mathcal{V}$ is the same as a linear map
\[
	B \rightarrow \E(X)\colon b \mapsto (b \ast -,-\ast b),
\]
such that $((bb')\ast-,-\ast(bb'))=\langle (b\ast -,-\ast b),(b' \ast -,- \ast b')\rangle$ and
\[
	\Phi_{k,i}(\alpha_1,\ldots,\alpha_k)=0, \quad i=1,\ldots,n,
\]
where $\alpha_1$, \dots, $\alpha_k$ are as in \Cref{prop_semidirect}.

We remark also that the bilinear map
\[ \langle -,- \rangle \colon \E(X) \times \E(X) \rightarrow \End(X)^2 \]
defines a partial operation $\langle -,- \rangle \colon \Omega \rightarrow \E(X)$, where $\Omega$ is the preimage
\[
	\langle -,- \rangle^{-1}(\E(X))
\]
of the inclusion $\E(X) \hookrightarrow \End(X)^2$.

Now we are ready to announce and prove our main result about the weak representability of actions of non-associative algebras.

\begin{theorem}\label{thmV}
	Let $\mathcal{V}$ be an action accessible, operadic variety of non-associative algebras over a field~$\F$ and let $U \colon \cV \rightarrow \PAlg$ denote the forgetful functor.
	\begin{enumerate}
		\item Let $X$ be an object of $\mathcal{V}$. There exists a monomorphism of functors
		      \[
			      \tau \colon \SplExt(-,X) \rightarrowtail \Hom_{\PAlg}(U(-),\E(X)),
		      \]
		      where $\SplExt(-,X)=\SplExt_\cV(-,X)$ and, for every object $B$ of $\cV$, $\tau_B$ is the injection which sends an element of $\SplExt(B,X)$ to the corresponding partial algebra homomorphism
		      \[
			      B \rightarrow \E(X)\colon b \mapsto (b \ast -, - \ast b).
		      \]
		\item Let $B$, $X$ be objects of $\cV$. The homomorphism of partial algebras
		      \[
			      B \rightarrow \E(X)\colon b \mapsto (b \ast -,- \ast b)
		      \]
		      belongs to $\Imm(\tau_B)$ if and only if $\Phi_{k,i}(\alpha_1,\ldots,\alpha_k)=0$, as in~\Cref{prop_semidirect}.
		\item If $(\E(X),\langle -,- \rangle)$ is an object of $\mathcal{V}$, then $(\E(X),\tau)$ becomes a weak representation of $\SplExt(-,X)$.
		\item If the diagram in $\V$ given by those subalgebras of $\E(X)$ which occur as a codomain of a morphism $(B \rightarrow \E(X)) \in \Imm(\tau_B)$ admits an amalgam in $\V$, then the colimit of that diagram determines an (initial) weak representation of $\SplExt(-,X)$.
		\item If $\cV$ is a variety of commutative or anti-commutative algebras, then $\E(X)$ is isomorphic to the partial algebra
		      \[
			      \lbrace f  \in \End(X) \mid \Phi_{k,i}(f,x_2,\ldots,x_k)=0, \; \forall x_2,\ldots,x_k \in X \rbrace
		      \]
		      endowed with the bilinear partial operation $\langle f, g \rangle = \alpha (f \circ g) + \beta (g \circ f)$, where $\alpha$, $\beta \in \F$ are given by the $\lambda/\mu$ rules.
	\end{enumerate}
\end{theorem}

Because of these results, we can give the following definitions.

\begin{definition}
	Let $X$ be an object of an action accessible, operadic variety of non-associative algebras $\cV$ with a choice of $\lambda/\mu$ rules. The partial algebra $\E(X)$ is called \emph{external weak actor} of $X$. The pair $(\E(X),\tau)$ is called \emph{external weak representation} of the functor $\SplExt(-,X)$. When $\tau$ is a natural isomorphism, we say that $\E(X)$ is an \emph{external actor} of $X$.
\end{definition}

\begin{proof}
	(1) The collection $\lbrace \tau_B \rbrace_B$ gives rise to a natural transformation since, for every algebra homomorphism $f\colon B' \rightarrow B$, the diagram in $\Set$
	\[
		\begin{tikzcd}
			\SplExt(B,X) \arrow[r, "\tau_{B}"] \arrow[d, "{\SplExt(f,X)}"]
			& \Hom(U(B),\E(X)) \arrow[d,"{\Hom(U(f),\E(X))}"]\\
			\SplExt(B',X) \arrow[r, "\tau_{B'}"] & \Hom(U(B'),\E(X))
		\end{tikzcd}
	\]
	where $\Hom(U(-),-)=\Hom_{\PAlg}(U(-),-)$, is commutative.	Moreover, for every object $B$ of $\cV$, the map $\tau_{B}$ is an injection, since every element of $\SplExt(B,X)$ is uniquely determined by the corresponding derived action of $B$ on $X$, i.e.\ by the pair of bilinear maps
	\[
		l \colon B \times X \rightarrow X,  \qquad r \colon X \times B \rightarrow X
	\]
	defined as in~\Cref{Rem:SplitV}. Thus $\tau$ is a monomorphism of functors.

	(2) Let $B$, $X$ be objects of $\cV$. A homomorphism of partial algebras $B \rightarrow \E(X)$ belongs to $\Imm(\tau_B)$ if and only if it defines a split extension of $B$ by $X$ in $\cV$. This is equivalent to saying that
	\[
		\Phi_{k,i}(\alpha_1,\ldots,\alpha_k)=0, \quad \forall i=1,\ldots,n,
	\]
	where $\alpha_1,\ldots,\alpha_k$ are as in \Cref{prop_semidirect}

	(3) If $(\E(X),\langle -,- \rangle)$ is an object of $\cV$, then we have a monomorphism of functors
	\[
		\tau \colon \SplExt(-,X) \rightarrowtail \Hom_{\cV}(-,\E(X)),
	\]
	and $(\E(X),\tau)$ is a weak representation of $\SplExt(-,X)$.

	(4) We may copy the ``if'' part of the proof of \Cref{Proposition Amalgam}, replacing the subalgebras of $\M(X)$ in $\CAAlg$ with those subalgebras of $\E(X)$ in $\V$ which occur as codomain of a morphism $(B \rightarrow \E(X)) \in \Imm(\tau_B)$. As in \Cref{reminitial}, by its construction as a colimit, the weak representation thus obtained is automatically an initial weak representation \cite{WAR}.

	(5) If $\cV$ is a variety of commutative (resp.\ anti-commutative) algebras, then for every object $X$ of $\cV$, $\E(X)$ consists of pairs of the form $(f \ast -, - \ast f)$ with $x \ast f = f \ast x$ (resp.\ $x \ast f = -f \ast x$), for every $x \in X$. Thus, an explicit isomorphism
	\[
		\lbrace f \in \End(X) \mid \Phi_{k,i}(f,x_2,\ldots,x_k)=0\rbrace \rightarrow \E(X)
	\]
	is given by $f \mapsto (f, \pm f)$.
\end{proof}

\begin{example}
	We may check that, with the \emph{obvious} choices of the $\lambda /\mu$ rules,
	\begin{enumerate}
		\item if $\mathcal{V}=\AbAlg$, then $\E(X) =0$ is the actor of $X$;
		\item if $\mathcal{V}=\CAAlg$, then $\E(X) \cong \M(X)$ is an external actor of $X$ (see \Cref{RemCAAlg});
		\item if $\cV=\JJAlg$, then as observed in \Cref{thmJJAlg}, the external actor $\E(X)$ is isomorphic to the partial algebra $\ADer(X)$ of anti-derivations of $X$;
		\item if $\mathcal{V}=\LieAlg$, then $\E(X)\cong \Der(X)$ is the actor of $X$;
		\item if $\mathcal{V}=\ACAntiAAlg$, then as observed in \Cref{thmAnti}, the external actor $\E(X)$ is isomorphic to the partial algebra $[X]$;
		\item if $\mathcal{V}=\Nil_2(\CAlg)$ or $\mathcal{V}=\Nil_2(\ACAlg)$, then $\E(X) \cong [X]_2$ is a weak actor of $X$;
		\item if $\cV=\Nil_2(\NAlg)$, then $\E(X)$ is an abelian algebra and it is a weak actor of $X$;
		\item if $\cV=\Alt$ over a field $\F$ with $\bchar(\F)\neq 2$ and $X$ is a unitary alternative algebra, then $\E(X) \cong X$ is an alternative algebra and we have a natural isomorphism
		      \[
			      \SplExt(-,X) \cong \Hom_{\Alt}(-,X)
		      \]
		      i.e.\ $X$ is the actor of itself. In particular, the algebra of octonions $\mathbb{O}$ has representable actions.
	\end{enumerate}
\end{example}

\begin{remark}
	The construction of the vector space $\E(X)$ can be done also in a variety of non-associative algebras $\cV$ which is not action accessible. However, there is no canonical way to endow $\E(X)$ with a bilinear map $\langle -,- \rangle$ as in~\Cref{E(X)} so we only have a monomorphism of functors
	\[
		\tau \colon \SplExt(-,X) \rightarrow \Hom_{\Vect}(U(-),\E(X)),
	\]
	where $U \colon \cV \rightarrow \Vect$ denotes the forgetful functor.
\end{remark}

\begin{remark}
	As described in \cite[Section~3]{CigoliManciniMetere}, for every Orzech category of interest~$\mathcal{C}$ and for every object $X$ of $\mathcal{C}$, it is possible to define a monomorphism of functors
	\[
		\mu \colon \SplExt(-,X) \rightarrowtail \Hom_{\mathcal{C}'}(V(-),\USGA(X)),
	\]
	where $\mathcal{C}'$ is a category which contains $\mathcal{C}$ as a full subcategory,$\USGA(X)$ is an object of $\mathcal{C}'$ called the \emph{universal strict general actor} of $X$ \cite{Casas} and $V \colon \C \rightarrow \C'$ denotes the forgetful functor. We further recall that $\USGA(X)$ is unique up to isomorphism, once the presentation of the Orzech category of interest $\C$ is fixed.

	For a variety of non-associative algebras $\cV$, a presentation is given by a choice of constants $\lambda_1$, \dots, $\lambda_8$, $\mu_1$, \dots, $\mu_8 \in \F$ which determine the $\lambda/\mu$ rules. In this case, it turns out that $\cV'=\NAlg$. Thus we have monomorphism of functors
	\[
		\mu \colon \SplExt(-,X) \rightarrowtail \Hom_{\NAlg}(V(-),\USGA(X))
	\]
	and, by~\Cref{thmV}, another monomorphism of functors
	\[
		\tau \colon \SplExt(-,X) \rightarrowtail \Hom_{\PAlg}(U(-),\E(X)).
	\]
	As explained at the beginning of \cite[Section 4]{Casas}, $\USGA(X)$ is the algebraic closure of the external weak actor $\E(X)$ with respect to the bilinear partial operation $\langle -,- \rangle$. When $\langle -,- \rangle$ is well defined on $\E(X) \times \E(X)$, then $\USGA(X)=\E(X)$ and $\mu=\tau$.
\end{remark}

However, it is often more convenient to work with the external weak actor $\E(X)$, since it is easier to construct than the universal strict general actor $\USGA(X)$. In fact, in the next section we shall present the construction of $\E(X)$ in different varieties of non-associative algebras.

\section{The quadratic case}\label{Section Quadratic}

In this section we introduce a systematic approach to finding the explicit structure of $\E(X)$ in the setting of operadic, quadratic varieties of algebras. Here we shall denote an element $(f\ast-,-\ast f)$ of $\E(X)$ by the symbol $f$; this means that $fx\coloneq f\ast x$ and $xf\coloneq x\ast f$.

Let $\cV$ be an action accessible, operadic, quadratic variety of non-associative algebras with no identities of degree $2$. Let us consider the free non-associative algebra generated by the symbols $f$, $x$ and $y$, and let us focus on its multilinear component of degree~$3$. There are $12$ possible monomials which we order as follows:
\begin{align*}
	f(xy) & > f(yx) > (xy)f > (yx)f
	> (fy)x > (fx)y                 \\ & > (yf)x > (xf)y
	> x(fy) > y(fx) > x(yf) > y(xf).
\end{align*}
Permuting the variables determines an action of the symmetric group~$\SSS_3$ on this space. For a given variety of algebras~$\cV$, we can write the orbit under the $\SSS_3$-action of its defining equations in matrix form, where each row corresponds to an equation and each column corresponds to a monomial, ordered as above. Let us denote this matrix by $M_3$, and its reduced row echelon form by $RM_3$. Action accessibility implies the following:

\begin{lemma}
	The rank of $M_3$ is at least $4$. Moreover, the $4 \times 4$ minor located on the top left of $RM_3$ is the identity matrix. \noproof
\end{lemma}

The vector space $\E(X)$ will be the subspace of $\End(X)^2$ formed by the pairs that satisfy the identities coming from $RM_3$.

Our task now is to endow this vector space with a partial multiplication, induced by action accessibility, and to provide strategies to check
\begin{enumerate}
	\item when this multiplication is total;
	\item when it induces a $\cV$-algebra structure on $\E(X)$.
\end{enumerate}

Let us rename the tags on the columns of $M_3$ by the following rule: $f \mapsto x$, $x \mapsto f$ and $y \mapsto g$. Then, the third and first columns of $RM_3$ will give us equations of the form
\begin{align*}
	(fg)x = \lambda_1 (fg)x & + \lambda_2 (fx)g + \lambda_3 (gf)x + \lambda_4 (xf)g                   \\
	                        & + \lambda_5 x(fg) + \lambda_6 y(fx) + \lambda_7 x(gf) + \lambda_8 g(xf) \\
	\intertext{and}
	x(fg) = \mu_1 (fg)x     & + \mu_2 (fx)g + \mu_3 (gf)x + \mu_4 (xf)g                               \\
	                        & + \mu_5 x(fg) + \mu_6 g(fx) + \mu_7 x(gf) + \mu_8 g(xf).
\end{align*}
At a first glance, these rules seem to yield a way of multiplying two elements $f$ and~$g$ belonging to $\E(X)$. However, this choice might not be unique. If the rank of $M_3$ is strictly larger than $4$, the lower rows will have zeroes in the first four positions, so adding any linear combination of them will produce a new bracket in~$\E(X)$.
Let us exemplify this with a concrete variety:

\begin{example}\label{Leibniz}
	The most common presentation of the variety of right Leibniz algebras is given by the identity $(xy)z - (xz)y - x(yz) = 0$. Then, $M_3$ will be the matrix
	\[
		\resizebox{\textwidth}{!}{$
				\begin{blockarray}{cccccccccccc}
					f(xy) & f(yx) & (xy)f & (yx)f & (fy)x & (fx)y & (yf)x & (xf)y & x(fy) & y(fx) & x(yf) & y(xf)
					\\
					\begin{block}{(cccccccccccc)}
						-1 & 0 & 0 & 0 & -1 & 1 & 0 & 0 & 0 & 0 & 0 & 0 \\
						0 & -1 & 0 & 0 & 1 & -1 & 0 & 0 & 0 & 0 & 0 & 0 \\
						0 & 0 & -1 & 0 & 0 & 0 & 0 & 1 & -1 & 0 & 0 & 0 \\
						0 & 0 & 1 & 0 & 0 & 0 & 0 & -1 & 0 & 0 & -1 & 0 \\
						0 & 0 & 0 & -1 & 0 & 0 & 1 & 0 & 0 & -1 & 0 & 0 \\
						0 & 0 & 0 & 1 & 0 & 0 & -1 & 0 & 0 & 0 & 0 & -1\\
					\end{block}
				\end{blockarray}
			$}
	\]
	while its reduced row echelon form is
	\[
		\resizebox{\textwidth}{!}{$
				\begin{blockarray}{cccccccccccc}
					f(xy) & f(yx) & (xy)f & (yx)f & (fy)x & (fx)y & (yf)x & (xf)y & x(fy) & y(fx) & x(yf) & y(xf) \\
					\begin{block}{(cccccccccccc)}
						1 & 0 & 0 & 0 & 1 & -1 & 0 & 0 & 0 & 0 & 0 & 0 \\
						0 & 1 & 0 & 0 & -1 & 1 & 0 & 0 & 0 & 0 & 0 & 0 \\
						0 & 0 & 1 & 0 & 0 & 0 & 0 & -1 & 0 & 0 & -1 & 0 \\
						0 & 0 & 0 & 1 & 0 & 0 & -1 & 0 & 0 & 0 & 0 & -1 \\
						0 & 0 & 0 & 0 & 0 & 0 & 0 & 0 & 1 & 0 & 1 & 0 \\
						0 & 0 & 0 & 0 & 0 & 0 & 0 & 0 & 0 & 1 & 0 & 1\\
					\end{block}
				\end{blockarray}
			$}
	\]
	Removing the rows in odd position---which we are entitled to, thanks to the obvious symmetry---we obtain that~$\E(X)$ is formed by the elements of $\End(X)^2$ satisfying the following identities:
	\begin{equation}\label{Leibeqs}
		\begin{aligned}
			f(xy) & = (fx)y - (fy)x \\
			(xy)f & = (xf)y + x(yf) \\
			x(fy) & = x(yf)
		\end{aligned}
	\end{equation}
	These are exactly the identities satisfied by biderivations.
	With the change of tag in the columns described before, we obtain
	\[
		\resizebox{\textwidth}{!}{$
				\begin{blockarray}{cccccccccccc}
					x(fg) & x(gf) & (fg)x & (gf)x & (xg)f & (xf)g & (gx)f & (fx)g & f(xg) & g(xf) & f(gx) & g(fx) \\
					\begin{block}{(cccccccccccc)}
						1 & 0 & 0 & 0 & 1 & -1 & 0 & 0 & 0 & 0 & 0 & 0 \\
						0 & 1 & 0 & 0 & -1 & 1 & 0 & 0 & 0 & 0 & 0 & 0 \\
						0 & 0 & 1 & 0 & 0 & 0 & 0 & -1 & 0 & 0 & -1 & 0 \\
						0 & 0 & 0 & 1 & 0 & 0 & -1 & 0 & 0 & 0 & 0 & -1 \\
						0 & 0 & 0 & 0 & 0 & 0 & 0 & 0 & 1 & 0 & 1 & 0 \\
						0 & 0 & 0 & 0 & 0 & 0 & 0 & 0 & 0 & 1 & 0 & 1\\
					\end{block}
				\end{blockarray}
			$}
	\]
	Therefore, the multiplication
	\begin{equation}\label{LeibMulti}
		\begin{aligned}
			(fg)x = (fx)g + f(gx) + \alpha_1 \big( f(xg) + f(gx) \big) + \alpha_2 \big( g(xf) + g(fx) \big) \\
			x(fg) = (xf)g - (xg)f + \beta_1 \big( f(xg) + f(gx) \big) + \beta_2 \big( g(xf) + g(fx) \big)
		\end{aligned}
	\end{equation}
	induces a partial algebra structure on $\E(X)$, for any choice of $\alpha_1, \alpha_2, \beta_1, \beta_2 \in \F$.
\end{example}

Now that we have a partial algebra structure induced on a general $\E(X)$, the next step is to verify when it is total. To do so, we have to focus on a partial subset of the set of consequences of degree~$4$. Let us consider the $120$-dimensional space formed by the multilinear monomials of degree~$4$ in the free non-associative algebra generated by the symbols $f, g, x, y$.
To gather all the consequences of the identities in degree~$3$, we have three different ways of operating. Let us take any identity from~$RM_3$.
The first way is to multiply it from the right or from the left by~$g$.
The second way, is to substitute $x$ by $(gx)$ or $(xg)$. Finally, we can substitute~$y$ by~$(gy)$ or by $(yg)$.
Doing all these substitutions together with the permutations of $f$ and~$g$, we obtain all the consequences. Note that in none of these identities the terms $(fg)$ or $(gf)$ will appear.

Now we need to check if the defining bracket satisfies the identities of $\E(X)$. To do so, we take again the identities from $RM_3$, substitute $f$ by $(fg)$ and expand it by the already defined product. The bracket will be closed if and only if these new obtained equations are linear combination of the previously obtained consequences.

To conclude, we shall check when the bracket satisfies the identities of the variety. This can be done just by directly substituting elements of $\E(X)$ in the defining equations of the variety. After applying them to a generic element $x$, once on the left and once on the right, it is a matter of substituting the bracket on $\E(X)$ when necessary.

\begin{example}
	Continuing with the Leibniz algebras example~\ref{Leibniz}, applying the procedure described before to the first equation in~\eqref{Leibeqs} yields:
	\begin{align*}
		g(f(xy)) & = g((fx)y) - g((fy)x), \\
		(f(xy))g & = ((fx)y)g - ((fy)x)g, \\
		f((gx)y) & = (f(gx))y - (fy)(gx), \\
		f((xg)y) & = (f(xg))y - (fy)(xg), \\
		f(x(gy)) & = (fx)(gy) - (f(gy))x, \\
		f(x(yg)) & = (fx)(yg) - (f(yg))x.
	\end{align*}
	It is a straightforward computation to check that the rank of the matrix formed by all the consequences is~$72$. Then, we have to compare it with the multiplication defined in~\Cref{LeibMulti}. For instance, taking again the first equation in~\Cref{Leibeqs} we have to expand the identity
	\[
		(fg)(xy) = ((fg)x)y - ((fg)y)x,
	\]
	which gives us
	\begin{align*}
		f(g(xy))+(f(xy))g + \alpha_1 \big(f(g(xy)) + f((xy)g)\big) + \alpha_2 \big(g(f(xy)) + g((xy)f)\big)   \\
		= (f(gx))y+((fx)g)y + \alpha_1 \big((f(gx))y + (f(xg))y\big) + \alpha_2 \big((g(fx))y + (g(xf))y\big) \\
		- (f(gy))x+((fy)g)x + \alpha_1 \big((f(gy))x + (f(yg))x\big) + \alpha_2 \big((g(fy))x + (g(yf))x\big).
	\end{align*}
	After a linear algebra computation it can be checked that no matter which $\alpha_1$ and~$\alpha_2$ we choose, it belongs to the subspace generated by the consequences. In fact, this will be true for all the identities~\eqref{Leibeqs}, so any $\alpha_1$, $\alpha_2$, $\beta_1$, $\beta_2 \in \F$ will produce a total multiplication on $\E(X)$.

	To check whether the induced bracket endows $\E(X)$ with a Leibniz algebra structure, we just need to check when the following identities hold
	\begin{align*}
		(f(gh))x & = ((fg)h)x - ((fh)g)x  \\
		x(f(gh)) & = x((fg)h) - x((fh)g).
	\end{align*}
	A quick computation tells us that this is only true when $( \alpha_1, \alpha_2, \beta_1, \beta_2 ) = (1, 0, 0, 0)$, so that we recover exactly the multiplication defined in~\cite[Definition~5.1]{Casas}.
\end{example}

We consider some further examples.

\begin{example}
	In the case of associative algebras any choice of bracket will induce a total algebra structure, but only the already known example of bimultipliers will be an associative algebra.
\end{example}

\begin{example}
	The variety of \emph{symmetric Leibniz algebras} is formed by the intersection between the varieties of right and left Leibniz algebras, i.e.\ the variety determined by $(xy)z-(xz)y-x(yz)=0$ (right Leibniz identity) and $z(xy)- (xz)y-x(zy)=0 $ (left Leibniz identity). The space generated by its bilinear identities of degree~$3$ has dimension~$10$, which means that there are $12$ parameters to define a product in $\E(X)$. With the help of a computer algebra system such as \texttt{Macaulay2}~\cite{M2} we check that any choice will give us a total algebra structure, and the set of variables that induces a symmetric Leibniz algebra structure on~$\E(X)$ forms an affine variety of dimension~$2$.
\end{example}

\begin{example}
	Following the algorithm proposed before, it can be checked easily that the variety of two-step nilpotent (non-commutative) algebras is weakly action representable. In fact, a weak actor may be given by the expected structure
	\[
		\E(X) = \{ f \in \End(X)^2 \mid f(xy)=(xy)f = 0 =(fx)y=x(yf)\}
	\]
	with product $fg= 0=gf$. Nevertheless, this is not the only product that can be induced. Since the space generated by its bilinear identities of degree~$3$ has maximum dimension~$12$, there are $16$ parameters that can be taken into account to define a product in $\E(X)$. All of them induce a total multiplication on it, and the set of parameters which induce a two-step nilpotent algebra on $\E(X)$ forms an affine variety of dimension~$3$. Note that these algebras were studied and classified in~\cite{LaRosaMancini1, LaRosaMancini2, LaRosaMancini3}.
\end{example}

\begin{example}
	Although commutative Poisson algebras are usually defined as a variety with two operations, in~\cite{MaRe} it was shown that with the depolarisation technique they can be seen as a quadratic variety (with one operation), so they fit in the scope of this section. The algorithmic approach presented before shows that it is possible to induce several total algebra structures on $\E(X)$, more precisely it gives rise to a $3$-parametric family.
\end{example}

\begin{example}\label{Example Novikov}
	The varieties of Novikov algebras or anti-associative algebras do not allow a total algebra structure on their respective $\E(X)$ induced by action accessibility, but it is still an open problem whether or not these varieties are weakly action representable.
\end{example}

\section{Open Questions and further directions}\label{Section Questions}

\subsection*{Converse of the implication ``\textit{weakly action representable category} $\Rightarrow$ \textit{action accessible category}''}

We studied the representability of actions of a general operadic variety of non-associative algebras over a field but we were not able to find an example of an action accessible variety which is not weakly action representable. Does the converse of the implication
\[
	\textit{weakly action representable category} \Rightarrow \textit{action accessible category}
\]
hold in this context?

\subsection*{Subvarieties}
We do not know how the condition (WRA) behaves under taking subvarieties (especially in the non-quadratic case, when the degree of the identities may be higher than $3$). For instance, we know that the variety $\AAlg$ is weakly action representable, but we do not know whether the subvariety $\Nil_k(\AAlg)$ for $k\geq 3$ satisfies the same condition. We recall that in this case, $\E(X)$ is an associative algebra, but it is not $k$-step nilpotent in general (see \Cref{Example Nilk}).

\subsection*{Initial weak representation}
As already mentioned in \Cref{reminitial}, a variety of non-associative algebras $\V$ is weakly action representable if and only if it is \emph{initially weakly action representable}, which means that for every object $X$, the functor $\SplExt(-,X)$ admits an initial weak representation. We do not know whether or not the weak representations that occur in this article when the external weak actor $\E(X)$ is an object of $\cV$ (such as for Leibniz algebras or associative algebras) are initial, or how we would check this in practice.

\subsection*{Representability of actions of \emph{unitary} algebras}

In the recent article \cite{IdeallyExact}, G.~Janelidze introduced the notion of \emph{ideally exact} category, with the aim of generalising semi-abelian categories in a way which includes relevant examples of \emph{non-pointed} categories, such as the categories $\textbf{Ring}$ and $\textbf{CRing}$ of (commutative) rings with unit.

A category $\C$ is \emph{ideally exact} when it is Barr-exact and Bourn-protomodular with finite coproducts, such that the unique morphism $0 \to 1$ in $\C$ is a regular epimorphism. Thus, semi-abelian categories are precisely the pointed ideally exact categories.

G.\ Janelidze also extended the notions of action representability and weak action representability to ideally exact categories, showing $\textbf{Ring}$ and $\textbf{CRing}$ are action representable, with the actor of a (commutative) unitary ring $X$ being isomorphic to~$X$ itself. We do not recall the construction here, since it is essentially the same as the one for alternative algebras with unit given in \Cref{exalt}.

We recall that a variety of non-associative algebras $\cV$ is said to be \emph{unitary closed} if for any object $X$ in it, the algebra $\tilde{X}$ spanned by $X$ and the element $1$, equipped with the multiplication $x \cdot 1 = 1 \cdot x = x$ for any $x \in X$, is still an object of~$\cV$. For instance, $\AAlg$ and $\Alt$ are unitary closed, and the category $\LeibAlg$, or any variety of anti-commutative algebras over a field of characteristic different from $2$, such as $\LieAlg$, are examples of varieties which are not. Thus, the condition of being unitary closed depends on the set of identities which determine the variety $\cV$.

When a variety of algebras $\cV$ is unitary closed, one can consider the subcategory~$\cV_1$ of unitary algebras of $\cV$ with the arrows being the algebra morphisms of~$\cV$ that preserve the unit. Of course, $\cV_1$ is an ideally-exact category and it is not pointed.

Examples \ref{exalt} and \ref{unitalassociative} suggest that for a unitary closed variety $\V$, one may use the construction of the external weak actor $\E(X)$ to study the representability of actions of the subcategory $\V_1$. For instance, it follows easily that $\Alt_1$ and $\AAlg_1$ are action representable, with the actor of an object $X$ in both cases being the object $X$ itself.

\section*{Acknowledgements}
First we would like to thank the anonymous referee, whose careful reading and invaluable comments helped us correct a crucial mistake in the first version of the text.

We are grateful to Abdenacer Makhlouf for recommending us to study the representability of actions of Jacobi--Jordan algebras, to Giuseppe Metere for suggesting us the name \emph{external weak actor}, and to Gabor P.\ Nagy for helping us with finding an explicit example of a $2$-Engel Lie algebra in characteristic $3$ which is not two-step nilpotent.

We would like to express our sincere gratitude to the Institut de Recherche en Mathématique et Physique (IRMP) for the warm reception we received during our visits to Louvain-la-Neuve. We would also like to extend our heartfelt appreciation to the Universities of Santiago de Compostela and Vigo for the generous support and welcoming atmosphere provided during our time there.

%\bibliography{tim}
%\bibliographystyle{amsplain-nodash}

\providecommand{\noopsort}[1]{}
\providecommand{\bysame}{\leavevmode\hbox to3em{\hrulefill}\thinspace}
\providecommand{\MR}{\relax\ifhmode\unskip\space\fi MR }
\providecommand{\MRhref}[2]{%
	\href{http://www.ams.org/mathscinet-getitem?mr=#1}{#2}
}
\providecommand{\href}[2]{#2}